\documentclass[11pt]{article} 
\usepackage{amsmath,latexsym,amssymb,mathtools,braket,bbm,color,pdfsync}
\usepackage{enumerate}
\usepackage{tikz}
\tikzset{bullet/.pic={\fill[radius=.3ex] circle;}}
 
\hsize \textwidth
\advance \hsize by -\marginparwidth
\evensidemargin
\oddsidemargin
\usepackage{graphicx}

\advance\hoffset by 5mm

\usepackage{booktabs} 
\usepackage{array} 
\usepackage{paralist} 
\usepackage{verbatim}  
\usepackage{subfig}  
\def\qed{\mbox{$\Box$}}
\newenvironment{proof}[1][Proof]{\par\noindent{\bf #1.\ }}{\hfill\qed\\ }
 
\usepackage[nottoc,notlof,notlot]{tocbibind} 
\usepackage[titles,subfigure]{tocloft} 
\newfam\msbmfam\font\tenmsbm=msbm10\textfont
\msbmfam=\tenmsbm\font\sevenmsbm=msbm7

\scriptfont\msbmfam=\sevenmsbm

\usepackage[pagebackref,colorlinks=true,pdfpagemode=UseNone,urlcolor=blue,linkcolor=blue,citecolor=blue]{hyperref}

\numberwithin{equation}{section}

\newtheorem{theorem}{Theorem}[section]
\newtheorem{lemma}[theorem]{Lemma}

\newtheorem{proposition}[theorem]{Proposition}

\newtheorem{thm}[theorem]{Theorem}

\newcommand{\one}{\mathbbm{1}}

\newcommand{\cX}{\mathcal{X}}

\newcommand{\Rm}{{\mathbb R}}
\newcommand{\Pm}{{\mathbb P}}
\def\1{\mathbf{1}}

\newcommand{\disp}{\displaystyle}

\newcommand{\be}{\begin{equation}}
\newcommand{\ee}{\end{equation}}
\newcommand{\bal}{\begin{aligned}}
\newcommand{\enbal}{\end{aligned}}

\newcommand{\pdr}[2]{\partial_{#2}{#1}}

\newcommand{\pdrr}[2]{\partial^2_{#2}{#1}}
\newcommand{\farc}{\frac}

\newcommand{\qe}{q_{\mathrm{e}}} 
\newcommand{\qm}{q_{\mathrm{m}}} 
\newcommand{\shift}{\mu} 
\newcommand{\wshift}{\nu} 
\newcommand{\earlyt}{t_-} 
\newcommand{\Ufactor}{A} 
\newcommand{\nonlin}{F} 
\newcommand{\earlyxi}{\xi_-} 
\newcommand{\coeff}{M} 
\newcommand{\rcoeff}{\widetilde{M}} 

\newcommand{\R}{\mathbb{R}}

\newcommand{\E}{\mathbb{E}}

\renewcommand{\P}{\mathbb{P}}

\newcommand{\eps}{\varepsilon}
\newcommand{\al}{\alpha}

\newcommand{\m}[1]{\mathcal{#1}}

\newcommand{\dn}{\mathrm{d}}
\newcommand{\ds}{\,\mathrm{d}}
\renewcommand{\d}{\;\mathrm{d}}

\newcommand{\der}[2]{\frac{\dn #1}{\dn #2}}

\newcommand{\abs}[1]{\left |#1 \right|}

\newcommand{\ti}[1]{\tilde{#1}}
\newcommand{\ubar}[1]{\underline{#1}}

\newcommand{\trm}[1]{\textup{\textrm{#1}}}

\newcommand{\anon}{\,\cdot\,}

\renewcommand{\And}{\quad \textrm{and} \quad}

\newcommand{\ForAll}{\quad \textrm{for all }}

\let\bar\smallbar

\title{The distance between the two BBM leaders}
\author{Julien Berestycki\footnote{Department of Statistics and Magdalen College, University of Oxford, UK; \href{mailto:julien.berestycki@stats.ox.ac.uk}{julien.berestycki@stats.ox.ac.uk}} 
 \and \'Eric Brunet\footnote{Laboratoire de Physique de l'\'Ecole normale sup\'erieure,
ENS, Universit\'e PSL, CNRS, Sorbonne Universit\'e, Universit\'e de Paris,
F-75005 Paris, France; \href{mailto:eric.brunet@ens.fr}{eric.brunet@ens.fr}}
  \and Cole Graham\footnote{Department of Mathematics, Stanford University, Stanford, CA 94350, USA; \href{mailto:grahamca@stanford.edu}{grahamca@stanford.edu}}
   \and Leonid Mytnik\footnote{Faculty of Industrial Engineering and Management, Technion, Technion City, Haifa 3200003, Israel; \href{mailto:leonid@ie.technion.ac.il}{leonid@ie.technion.ac.il}}
\and Jean-Michel Roquejoffre\footnote{Institut de Math\'ematiques de Toulouse; UMR 5219, Universit\'e de Toulouse; CNRS, UPS IMT, F-31062 Toulouse Cedex 9, France; E-mail: \href{mailto:Jean-Michel.Roquejoffre@math.univ-toulouse.fr}{Jean-Michel.Roquejoffre@math.univ-toulouse.fr}}
\and Lenya Ryzhik\footnote{Department of Mathematics, Stanford University, Stanford, CA 94350, USA; \href{mailto:ryzhik@stanford.edu}{ryzhik@stanford.edu}}
}

\begin{document}

\maketitle
\medskip

\begin{abstract}
  We study the distance between the two rightmost particles in branching Brownian motion.
  Derrida and the second author have shown that the long-time limit $d_{12}$ of this random variable can be expressed in terms of PDEs related to the Fisher--KPP equation.
  We use such a representation to determine the sharp asymptotics of $\P(d_{12} > a)$ as $a\to+\infty$.
  These tail asymptotics were previously known to ``exponential order;'' we discover an algebraic correction to this behavior.  
\end{abstract}

\section{Introduction}

We consider a branching Brownian motion (BBM) that starts with a single particle at 
the position~$x=0$ at time $t=0$, with 
branching occurring at rate $1$. 
At each branching event,
the particle splits into a random number of particles.
Let $p_k$ denote the probability of $k$ offspring in a given branching event.
We assume that $p_0 = p_1 = 0$, so
\begin{equation}
  \label{20mar2202}
  \sum_{k=2}^\infty p_k=1.
\end{equation}
We let $N$ denote the expected number of offspring:
\begin{equation*}
  N=\sum_{k= 2}^\infty kp_k < \infty,
\end{equation*}
and we assume that the offspring distribution has a higher moment:
\begin{equation}
  \label{eq:higher-moment}
  \sum_{k= 2}^\infty k^{1 + \beta} p_k <\infty \quad \textrm{for some } \beta \in (0, 1).
\end{equation}
To simplify the diffusion
constant in the Fisher--KPP equation below, we assume that the variance of each
individual Brownian motion is $2t$ (rather than~$t$).
Let $x_1(t)\geq x_2(t)\geq \ldots\ge x_{n(t)}$ denote the positions of the
particles that exist at time $t\ge 0$. Here, $n(t)$ is the total number of particles 
present at time~$t$. Note that $n(t)$ is itself random.
It is well known from the work of Bramson in~\cite{Bramson1,Bramson2}, followed by the 
paper by Lalley and Sellke~\cite{LS}, with a shorter proof in a more recent paper by Roberts~\cite{Roberts},
that the maximal particle $x_1(t)$
is typically at distance $O(1)$ from the position
\begin{equation}
  \label{20mar2206}
m(t)= c_*t -\frac{3}{2\lambda_*} \log (t+1)
\end{equation}
as $t \to \infty$, with 
\begin{equation}
  \label{20mar2302}
c_*=2\sqrt{N-1},\qquad\lambda_*=\sqrt{N-1}.
\end{equation}
This result does not depend on the precise nature of branching, so that both the spreading speed $c_*$ 
and the logarithmic correction in (\ref{20mar2206}) depend only on the  expected number of offspring $N$.  

To see the connection between Bramson's results and PDEs, recall that
McKean showed in~\cite{McK} that the cumulative distribution function of the maximal particle
\begin{equation}\label{20mar1902}
H(t,x)=\Pm[x_1(t)\ge x]
\end{equation}
satisfies the Fisher--KPP equation \cite{Fisher, KPP}
\begin{equation}\label{20mar1904}
\pdr{H}{t}=\pdrr{H}{x}+f(H),
\end{equation}
with the initial condition
\begin{equation}\label{20mar1906}
H(0,x)= \one_{(-\infty, 0]}(x).
\end{equation}
The nonlinear reaction in (\ref{20mar1904}) has the form
\begin{equation}\label{20mar2208}
f(u)=1-u-\sum_{k={2}}^\infty p_k (1-u)^k.
\end{equation}

Bramson's probabilistic results imply that there is a constant $\bar x_0 \in \Rm$ such that
\begin{equation}\label{20mar1908}
H\big(t,x+m(t)\big)\to U (x-\bar x_0)~~\hbox{ as $t\to+\infty$}
\end{equation} 
uniformly on $\R$.
Here, $U(x)$ is a Fisher--KPP traveling wave, so that $U(x-c_*t)$ is
a solution to (\ref{20mar1904}) moving with speed $c_*$. The function
$U(x)$ satisfies
\begin{equation}\label{20mar1910}
-c_*U' = U''+f(U),~~U(-\infty)=1,~~U(+\infty)=0
\end{equation}
and has the asymptotics
\begin{equation}
  \label{eq:wave-limits}
  \begin{aligned}
    U(x) &\sim x e^{-\lambda_*x} \quad &&\trm{as } x \to +\infty,\\
    U(x) &\sim 1-\Ufactor e^{\gamma_*x} \quad &&\trm{as } x \to -\infty
  \end{aligned}
\end{equation}
for
\begin{equation*}
  \gamma_*=\sqrt{N}-\sqrt{N-1}
\end{equation*}
and some $\Ufactor>0$.
Any translate of a solution to (\ref{20mar1910}) is also a solution.
However, there is a unique shift such that the pre-factor in front of the right side of \eqref{eq:wave-limits} is one.
This, in turn, identifies the constant $\Ufactor$.

More generally, we have the following PDE result originating in work of Bramson~\cite{Bramson1,Bramson2} and Uchiyama~\cite{Uchiyama} and later developed in~\cite{Graham,HNRR,Lau,NRR1,NRR2}.
Let $u(t,x)$ solve~(\ref{20mar1904}),
\begin{equation}\label{20mar1914}
\pdr{u}{t}=\pdrr{u}{x}+f(u),
\end{equation}
with initial condition
\begin{equation*}
u(0,x)=\phi(x)
\end{equation*}
for a non-negative bounded function  $\phi \not\equiv 0$ which is compactly supported on the right.
That is, there exists $L_0>0$ so that~$\phi(x)=0$ for all $x\ge L_0$.
Assume that $f \in \m{C}^{1, \beta}([0, 1])$ satisfies the Fisher--KPP assumptions
\begin{equation}
  \label{eq:KPP-type}
  f(0)=f(1)=0, \quad f|_{(0, 1)}>0, \quad \trm{and} \quad f(u)\le f'(0)u \enspace \trm{for all } u \in (0, 1).
\end{equation}
Then, there exists a constant
$s[\phi]\in\Rm$ (depending also on $f$), known as the  Bramson shift corresponding to $\phi$,
such that
\begin{equation}
  \label{20mar1918}
  u\big(t,x+m(t)\big)\to U (x-s[\phi]) \quad \trm{as } t\to+\infty
\end{equation} 
uniformly on compact sets, with $c_*=2\sqrt{f'(0)}$ and $\lambda_*=\sqrt{f'(0)}$ in the definition (\ref{20mar2206}) of $m(t)$.
In particular, (\ref{20mar1908}) states that
\begin{equation}\label{20mar2512}
s\left[\one_{(-\infty, 0]}\right]=\bar x_0.
\end{equation}

It is straightforward to check that $f(u)$ defined in (\ref{20mar2208}) satisfies the Fisher--KPP property~\eqref{eq:KPP-type}.
Indeed, it is immediate from (\ref{20mar2208}) and \eqref{20mar2202} that
\begin{equation*}
  f(0) = f(1) = 0.
\end{equation*}
Moreover,
\begin{equation}
  \label{eq:f-deriv-endpts}
  f'(0) = -1 + \sum_{k=2}^\infty k p_k = N - 1 \And f'(1) = -1.
\end{equation}
Finally, $f'' \leq 0$ on $[0, 1]$, so $f|_{(0, 1)} > 0$ and $f(u) \leq f'(0) u$ on $[0, 1]$.
Moreover, the moment bound \eqref{eq:higher-moment} ensures that $f \in \m{C}^{1, \beta}([0, 1])$.

Let us mention that not every nonlinearity of the Fisher--KPP type
comes from a branching process, and that convergence in (\ref{20mar1918}) holds for all
nonlinearities of the Fisher--KPP type, not just ``probabilistic'' ones. 
One reflection of the rigidity of the class of probabilistic nonlinearities is that they satisfy
\begin{equation}\label{20mar2430}
f'(1)=-1,
\end{equation}
and the Fisher--KPP property is not related to $f'(1)$.
Also, probabilistic nonlinearities are concave, which is not necessary for~\eqref{eq:KPP-type}.

For later use, we introduce the decomposition
\begin{equation}\label{20mar2434}
f(u)=(N-1)u-\nonlin(u),
\end{equation}
so that $\nonlin$ denotes the ``nonlinear'' part of the reaction $f$.
By \eqref{eq:f-deriv-endpts} and \eqref{eq:KPP-type}, $\nonlin'(0)=0$ and $\nonlin(u)> 0$ for all $u\in(0,1]$.
When $f$ is probabilistic, $f'' \leq 0$ and \eqref{20mar2430} imply that $\nonlin$ and $\nonlin'$ are increasing and
\begin{equation}
  \label{eq:nonlin-deriv}
  0 = F'(0) < F'(u) < F'(1) = N \quad \trm{for all } u \in (0, 1).
\end{equation}

The connection between branching Brownian motion and the Fisher--KPP equation runs deeper than \eqref{20mar1902} and \eqref{20mar1904}.
Consider the measure-valued process that characterizes the BBM seen from the position $m(t)$:
\begin{equation}
  \label{20mar2310}
  \cX_t = \sum_{k=1}^{ n(t)} \delta_{m(t)-x_k(t)},
\end{equation}
recalling that $n(t)$ is the total number of particles alive at time $t$ and
$x_1(t)\ge x_2(t)\ge\cdots\ge x_{n(t)}(t)$ are the positions of those particles.
In the binary case $p_2 = 1$, \cite{ABBS,ABK2,BD2} show that
\begin{equation}
  \label{eq:extremal}
  \cX_t\xrightarrow{\trm{law}} \cX=\sum_{k \geq 1} \delta_{\chi_k} \quad \trm{as } t \to \infty.
\end{equation}
We use the convention $\chi_1\leq \chi_2\leq \ldots$
Due to the sign reversion in (\ref{20mar2310}), $\chi_1$ corresponds to the rightmost particle in the original process, $\chi_2$ to the second largest, and so on.
In the literature,~$\cX$ is often called the limiting extremal process (see \cite{ABK1}, \cite{ABK2}).

In this paper, we examine the distance between the two leading particles $x_1(t)$ and $x_2(t)$.
More precisely, we study the distance after the limit $t\to+\infty$:
\begin{equation*}
d_{12} ={\chi_2 - \chi_1}.
\end{equation*} 
The main result of this paper is the tail asymptotics of the random variable $d_{12}$.
\begin{thm}
\label{thr:main}
As $a \to \infty$, we have
\begin{equation}
  \label{eq:second-particle}
  \P(d_{12} > a) = \frac{\Ufactor\gamma_*}{2\lambda_*^2\sqrt{\pi}} 
  \left(\frac{a}{2\sqrt{N}}\right)^{{3\sqrt{N}/2}} e^{-\big(\sqrt{N} + \sqrt{N-1}\big) (a+\bar x_0)} 
  \left[1 + O\left(a^{-1/2}\right)\right]
\end{equation}
with $\Ufactor$ as in (\ref{eq:wave-limits}). 
\end{thm}
To prove Theorem~\ref{thr:main}, we relate the distance $x_1(t) - x_2(t)$ to a derivative of the Fisher--KPP equation with respect to its initial data.
The bulk of the proof is a detailed study of this derivative.

Weaker versions of the above result were already known.
In~\cite{BD2}, Derrida and the second author predicted the exponential order of \eqref{eq:second-particle} for binary BBM.
Cortines, Hartung and Louidor confirmed this conjecture \cite[Theorem 1.4]{CHL}, again in the binary case $N = 2$:
\begin{equation*}
  \lim_{a\rightarrow \infty} a^{-1} \log \P(d_{12} > a) = -\big(\sqrt{2} + 1\big).
\end{equation*}
(The limit given in~\cite{CHL} is in fact $-\big(\sqrt{2} + 2\big)$, since the authors use Brownian motion of variance $t$ rather than $2t$.)
However, neither \cite{BD2} nor \cite{CHL} discuss the sub-exponential behavior of $\P(d_{12} > a)$.
In particular, the algebraic pre-factor $a^{{3}\sqrt{N}/2}$ in Theorem~\ref{thr:main} is new.
We discuss the origin of this peculiar term in the next section.

More broadly, \cite{BD2} is one of the main motivations for our study.
It lays out numerous fascinating connections between the Bramson shifts of certain Fisher--KPP solutions and fine properties of branching Brownian motion.
Some of these predictions have been addressed elsewhere~\cite{MRR}, but many have not yet been proven rigorously.

We close with a word on notation: throughout the paper, we let $C > 0$ and $c > 0$ be constants which may change from line to line.
We think of $C$ as large and $c$ as small.

\section*{Acknowledgments}
CG was partially supported by the Fannie and John Hertz Foundation and NSF grant DGE-1656518,
LM and LR by the US-Israel Binational Foundation, and LR by NSF grants DMS-1613603 and DMS-1910023.

\section{Outline of the proof of Theorem~\ref{thr:main}}

We begin by relating the law of $d_{12}$ to the long-time behavior of a certain PDE solution.

\subsection*{An expression for the distribution of $d_{12}$}

Consider a branching Brownian motion shifted to start from a single particle at position $-a$ rather than $0$.
Recall that $x_1(t)\ge x_2(t)\ge\cdots\ge x_{n(t)}(t)$ are the positions of the $n(t)$ particles alive at time $t$.
We introduce the (unnormalized) density
\be
\label{defz}
z(t,x;a) \d x:=\P\big(x_1(t)\in [x, x + \dn x], \,\ x_2(t)<x-a\big)
\ee
under the convention that $x_2(t)=-\infty$ before the first branching event.
We derive an evolution equation for $z$ by considering the dependence of $z(t+\eps,x;a)$ on the state of the system at the small time $\eps > 0$.

At time $\eps$, there have been no branching events with probability $1-\eps+o(\eps)$,
one branching event with probability $\eps+o(\eps)$,
and more than one branching event with probability $o(\eps)$.
We will only track terms of order $\eps$ or larger, so the last event can be safely discarded.
In the first case, there is a single particle present at time $\eps$.
Its position is $-a+\eta\sqrt\eps$, where $\eta$ is a centered Gaussian of variance 2.
By the Markov property, the conditional probability that the event defining $z(t+\eps,x;a)$ is fulfilled is then $z(t,x-\eta\sqrt\eps;a)$.

In the second case, there are $k$ particles (with $k>1$ random) with positions $-a+O(\sqrt\eps)$ at time $\eps$.
Given this, we claim that the probability that the event defining $z(t+\eps,x;a)$ is fulfilled is
\begin{equation}
  \label{eq:one-branching}
  k z(t,x;a)[1-H(t,x)]^{k-1}+o(1).
\end{equation}
Indeed, the descendant-BBM generated by one particle should satisfy the event defining
\begin{equation*}
  z\big(t, x + O(\sqrt{\eps}); a\big) = z(t, x; a) + o(1),
\end{equation*}
while the descendant-BBMs generated by the remaining $k - 1$ particles should \emph{all} lie to the left of $x - a$ at time $t + \eps$.
Each of these $k - 1$ particles began near $-a$ at time $\eps$, so the probability that its descendant-BBM lies to the left of $x - a$ at time $t + \eps$ is $1 - H(t, x) + o(1)$.
Finally, the particles evolve independently after time $\eps$ and there are $k$ combinatorial choices for the distinguished $z$-particle.
The product formula \eqref{eq:one-branching} follows.

Gathering these events, we conclude that
\begin{equation*}
  \bal
  z(t+\eps,x;a)
  &=\E\big[(1-\eps)  z(t,x-\eta\sqrt\eps;a)+\eps k  z(t,x;a)
  [1-H(t,x)]^{k-1}+o(\eps)\big]
  \\
  &=(1-\eps) z(t,x;a)+\partial_x^2  z(t,x;a)\eps+\eps  z(t,x;a)\sum_{k\ge2} p_k
  k [1-H(t,x)]^{k-1}+o(\eps),
  \enbal
\end{equation*}
where the expectation in the first line is taken over $\eta$ and $k$.
Taking $\eps\searrow0$ and recalling the definition of $f$ in \eqref{20mar2208}, we obtain
\begin{equation}
  \label{eq:z-evolution}
  \pdr{z(t,x;a)}{t}=\pdrr{z(t,x;a)}{x} + f'(H(t,x))z(t,x;a),\quad z(0,x;a)=\delta(x+a).
\end{equation}
By the definition of $z$,
\begin{equation*}
\P\big(x_1(t)-x_2(t)>a\big) = \int_\R z(t, x; a) \d x.
\end{equation*}
Taking $t\to\infty$:
\begin{equation}
  \label{eq:law-limit}
  \P(d_{12} > a) = \lim_{t \to \infty} \int_\R z(t, x; a) \d x.
\end{equation}

To understand the long-time behavior of $z$, we observe that \eqref{eq:z-evolution} resembles a derivative of the Fisher--KPP equation \eqref{20mar1904}.
Indeed $\partial_x H$ solves \eqref{eq:z-evolution} with different initial data.
Bramson showed that $H$ converges to a traveling wave in the $m$-moving frame.
If we apply standard parabolic estimates to \eqref{20mar1908}, we can conclude that
\begin{equation*}
  \partial_x H(t, x + m(t)) \to U'(x - \bar{x}_0)
\end{equation*}
uniformly in $x \in \R$ as $t \to \infty$.
This suggests that $U'(x - \bar{x}_0)$ is a stable steady state at $t = \infty$ of \eqref{eq:z-evolution} in the moving frame.
Since \eqref{eq:z-evolution} is linear, we might expect that $z$ converges to a multiple of this state in the moving frame.
The methods of \cite{HNRR, NRR1} can be adapted to confirm these heuristics:
\begin{proposition}
  \label{prop:z-limit}
  For each $a > 0$, there exists $\coeff(a) > 0$ such that
  \begin{equation}
    \label{eq:z-limit}
    z(t, x + m(t); a) \to -\coeff(a) U'(x - \bar{x}_0)
  \end{equation}
  uniformly in $x \in \R$ as $t \to \infty$.
\end{proposition}
This proposition is essentially a consequence of Proposition~4.1 and Lemma~5.1 in~\cite{NRR1}; we omit the proof.
Those results concern solutions to the Fisher--KPP equation, but \eqref{eq:z-evolution} has a similar structure and the arguments adapt without difficulty.

To use \eqref{eq:z-limit} in \eqref{eq:law-limit}, we need to commute the long-time limit and the spatial integral.
\begin{lemma}
  \label{lem:limit-integral}
  For all $a > 0$,
  \begin{equation}
    \label{eq:limit-integral}
    \lim_{t \to \infty} \int_\R z(t, x + m(t); a) \d x = \coeff(a).
  \end{equation}
\end{lemma}

\begin{proof}
  Throughout, we fix $a > 0$.
  By the comparison principle and \eqref{eq:z-evolution}, $z \geq 0$.
  Since \eqref{eq:z-limit} is uniform in $x$, it suffices to show that none of the mass of $z$ escapes to infinity.
  Recalling the definition \eqref{defz}, we have
  \begin{equation*}
    z(t, x; a) \leq \P\big(x_1(t) \in [x, x + \dn x]\big) = -\partial_x H(t, x + a).
  \end{equation*}
  Hence for $L > 0$, \eqref{20mar1908} and \eqref{eq:wave-limits} imply that
  \begin{equation*}
    \int_L^\infty z(t, x + m(t); a) \d x \leq H(t, L + m(t) + a) \leq 2 U(L + a - \bar{x}_0) \leq C(a) e^{-cL},
  \end{equation*}
  provided $t$ is sufficiently large (depending on $L$ and $a$).
  Similarly, \eqref{eq:wave-limits} yields
  \begin{equation*}
    \int_{-\infty}^{-L} z(t, x + m(t); a) \d x \leq C(a) e^{-cL}
  \end{equation*}
  when $t$ is large.
  Therefore \eqref{eq:z-limit} implies
  \begin{equation*}
    \begin{aligned}
        \lim_{t\to\infty}\int_\R z(t,x + m(t);a) \d x &= \lim_{t\to\infty} \int_{-L}^L z(t, x + m(t);a) \d x + O\big(e^{-cL}\big)\\
        &= \coeff(a) \Big[U(-L-\bar x_0)-U(L-\bar x_0)\Big] + O\big(e^{-cL}\big)\\
        &= \coeff(a) + O\big(e^{-cL}\big),
    \end{aligned}
  \end{equation*}
  where we have again used \eqref{eq:wave-limits}.
  Since $L > 0$ was arbitrary, \eqref{eq:limit-integral} follows.
\end{proof}

Combining \eqref{eq:law-limit} and \eqref{eq:limit-integral}, we obtain $\P(d_{12} > a) = \coeff(a)$.
Hence Proposition~\ref{prop:z-limit} becomes
\begin{equation}
  \label{eq:law-coeff}
  z(t, x + m(t); a) \to - \P(d_{12} > a) U'(x - \bar{x}_0) \quad \trm{as } t \to \infty
\end{equation}
uniformly in $x \in \R$.
We will use this long-time limit to determine the the dependence of $\P(d_{12} > a)$ on $a$.
First, however, we describe a connection with the Bramson shift introduced in \eqref{20mar1918}.

\subsection*{The distribution of $d_{12}$ through Bramson shifts}

As noted above, \eqref{eq:z-evolution} appears to be a derivative of \eqref{20mar1904}.
Precisely, it is the derivative with respect to a certain perturbation of the initial data $\one_{(-\infty, 0]}$.
Consider the initial condition
\begin{equation}
  \label{eq:initial-y}
  u(0, x; y, a) = \one_{(-\infty, 0]}(x) - \one_{(y - a, -a]}(x)
\end{equation}
parameterized by $y < 0$ and $a > 0$.
We emphasize that $y$ is negative, so $y - a < -a$.
We think of $a$ as fixed, but allow $y$ to vary.
Note that
\begin{equation}
  \label{eq:y=0}
  u(0, x; 0, a) = \one_{(-\infty, 0]}(x)
\end{equation}
for all $a > 0$ and
\begin{equation}
  \label{eq:chi-deriv}
  \partial_y u(0, x; y, a) = \delta(x - y + a)
\end{equation}
in the distributional sense.
Now let $u(t, x; y, a)$ denote the solution to the Fisher--KPP equation \eqref{20mar1914} with initial data \eqref{eq:initial-y}.
In particular, \eqref{20mar1906} and \eqref{eq:y=0} imply that
\begin{equation}
  \label{eq:H-formula}
  u(t, x; 0, a) = H(t, x)
\end{equation}
for all $a > 0$.

Now, standard parabolic estimates imply that $u$ is differentiable in $y$ and that its derivative satisfies the appropriate derivative of \eqref{20mar1914}.
That is, if
\begin{equation}
  \label{eq:z-def-extended}
  z(t, x; y, a) = \partial_y u(t, x; y, a),
\end{equation}
then \eqref{eq:chi-deriv} yields
\begin{equation*}
  \partial_t z(t, x; y, a) = \partial_x^2 z(t, x; y, a) + f'\big(u(t, x; y, a)\big) z(t, x; y, a), \quad z(0, x; y, a) = \delta(x - y + a).
\end{equation*}
In particular, \eqref{eq:H-formula} and uniqueness imply that $z(t, x; 0, a) = z(t, x; a)$.
That is, the distribution of $x_1(t) - x_2(t)$ can be expressed via derivatives of the Fisher--KPP equation with respect to its initial data.
This connection allows us to relate the law of $d_{12}$ to the Bramson shift in \eqref{20mar1918}.

Recall that for each fixed $y$ and $a$, \eqref{20mar1918} implies that
\begin{equation} 
  \label{eq:y-shift}
  u(t, x + m(t); y, a) \to U(x - s(y, a))
\end{equation}
uniformly in $x$ as $t\to\infty$,
where $s(y,a)=s\big[u(0,x;y,a)\big]$ is the Bramson shift associated with the initial condition $u(0,x;y,a)$.
This is well defined because $u(0, x; y, a)$ is bounded, nonnegative, and compactly supported on the right.
In Appendix~\ref{sec:s-differentiable}, we show that $s$ is differentiable in $y$ and that $\partial_y$ commutes with the long-time limit \eqref{eq:y-shift}.
Hence \eqref{eq:z-def-extended} and \eqref{eq:y-shift} imply
\begin{equation}
  \label{eq:z-limit-formal}
  \begin{aligned}
    \lim_{t \to \infty} z(t, x + m(t); a) &= \lim_{t \to \infty} \partial_y u(t, x + m(t); y, a)\big|_{y = 0}\\
    &= \partial_y \left. \left(\lim_{t \to \infty} u(t, x + m(t); y, a)\right)\right|_{y = 0}\\
    &= \partial_y U(x - s(y, a))\big|_{y = 0} = -\partial_y s(0, a) U'(x - \bar{x}_0).
  \end{aligned}
\end{equation}
In the last line, we used the identity $s(0,a)=\bar x_0$, which follows from \eqref{20mar2512} and \eqref{eq:y=0}.
Comparing with \eqref{eq:law-coeff}, we see that
\begin{equation}
  \label{eq:shift-deriv}
  \P(d_{12} > a) = \partial_y s(0, a).
\end{equation}
That is, the gap between the two leaders in the BBM extremal process is encoded in the Bramson shifts $s(y, a)$.

This can be seen directly from ideas of \cite{BD2} and \cite{MRR}.
Using \eqref{eq:initial-y} and McKean's identity \cite{McK}, we have
\begin{equation*}
\begin{aligned}
1-u(t,x+m(t);y,a)&=\E\left[\prod_{k=1}^{n(t)}\Big(1-u(0,x+m(t)-x_k(t);y,a)\Big)\right]
\\
  &=\P\Big[x+m(t)-x_k(t)\in (y-a,-a]\cup(0,\infty)\quad\trm{for all } 1\le k\le n(t)\Big]
\\
  &=\P\Big[m(t)-x_k(t)\in (y-a-x,-a-x]\cup(-x,\infty)\quad\trm{for all } 1\le k\le n(t)\Big].
\end{aligned}
\end{equation*}
Take the limit $t\to\infty$, \eqref{eq:y-shift} and \eqref{eq:extremal} imply
\begin{equation}
  \label{eq:long-time-McKean}
\begin{aligned}
  1-U(x-s(y,a)) &= \P\Big[\chi_k\in (y-a-x,-a-x]\cup(-x,\infty)\quad\trm{for all } k\ge1\Big]\\
  &= \P\big[\chi_1>-x\big]+\P\big[\chi_1\in(y-a-x,-a-x],\, \chi_2>-x\big]+O(y^2),
\end{aligned}
\end{equation}
where we used $\chi_1\le\chi_2\le\ldots$.
The $O(y^2)$ term contains all events with multiple particles in $(y-a-x,-a-x]$.
The first term in the right side of \eqref{eq:long-time-McKean} is $1-U(x-\bar x_0)=1-U(x-s(0,a))$.
Moving this term to the left side and dividing by $\abs{y}$, we
obtain in the $y\nearrow0$ limit
\begin{equation*}
  \begin{aligned}
    \partial_y U(x - s(y, a))\big|_{y = 0} = -\partial_y s(0, a) U'(x
    - \bar{x}_0) &= \lim_{y\nearrow0}\frac{\P\big[\chi_1\in(y-a-x,-a-x], \, \chi_2>-x \big]}{\abs{y}}\\
    &= \frac{1}{\dn x}\P\big[-(\chi_1 + a) \in [x, x + \dn x), \; d_{12} > a\big].
  \end{aligned}
\end{equation*}
Integrating over $x$, we recover \eqref{eq:shift-deriv}.
Thus, our work adds to a growing body of literature identifying properties of the extremal process with long-time behavior of the Fisher--KPP equation.
For results of a similar flavor, see \cite{BD2} and \cite{MRR}.

\subsection*{Long-time asymptotics for $z$}

The previous subsection puts our problem in a broader context, but it is not essential to our proof of Theorem~\ref{thr:main}.
Instead, we can proceed directly from \eqref{eq:law-coeff}, which states that $\P(d_{12} > a)$ is the coefficient of $-U'$ in the long-time limit of $z$ in the $m$-moving frame.
We must determine the dependence of this coefficient on $a$ as $a \to \infty$.
To do so, it helps to shift by $m$ and remove the exponential decay of $U'$.
We  define
\begin{equation}
  \label{20mar2508}
  r(t,x)= (t+1)^{-3/2} e^{\lambda_*(x + a)} z(t,m(t)+x;a).
\end{equation}
The factor of $(t+1)^{-3/2}$ above is a matter of convenience which will be explained below.
Of course, $r$ depends on $a$ as well, but explicit dependence will become cumbersome in the remainder of the paper, so we drop it.

Using 
(\ref{20mar2206}) and (\ref{20mar2302}), we see that the function $r(t,x)$ satisfies
\begin{equation}\label{eq:main1}
\pdr{r}{t} = \pdrr{r}{x} - \frac{3}{2\lambda_*(t+1)} \pdr{r}{x}
- \Big[N-1-f'\big(H(t,m(t)+x)\big)\Big] r, 
\quad r(0,x) = \delta(x+a).
\end{equation}
Recalling the  decomposition
\eqref{20mar2434} of $f$, we see that the coefficient of the reaction term is $-\nonlin'(H)$.
For concision, we introduce the notation
\begin{equation}
  \label{20may1214}
  V(t, x) = \nonlin'(H(t,m(t)+x))
\end{equation}
and write \eqref{eq:main1} as
\begin{equation}\label{eq:main}
\pdr{r}{t} = \pdrr{r}{x} - \frac{3}{2\lambda_*(t+1)} \pdr{r}{x} - V(t, x) r, 
\quad r(0,x) = \delta(x+a).
\end{equation}
By \eqref{eq:nonlin-deriv}, $0 < V < N$ and
\begin{equation*}
  \begin{aligned}
    V(t,x)&\to N \quad &&\trm{as }x\to-\infty,\\
    V(t,x)&\to 0 \quad &&\trm{as } x\to+\infty.
  \end{aligned}
\end{equation*}
Moreover,
\begin{equation}
  \label{eq:potential-infty}
  V(t,x) \to V(\infty,x) = \nonlin'\big(U(x-\bar x_0)\big) \quad \trm{ as } t\to\infty.
\end{equation}

The potential $V$ largely confines $r$ to $\R_+$.
Since the drift term in \eqref{eq:main} decays in time, we can thus view \eqref{eq:main} as a heat equation on $\R_+$ with (approximate) Dirichlet conditions at $x = 0$.
It follows that the dynamics of $r$ are largely driven at the diffusive scale $x \asymp \sqrt{t}$, by which we mean $c \sqrt{t} \leq x \leq C \sqrt{t}$.
In particular, at that scale we should expect $r$ to resemble a multiple of the fundamental solution to the Dirichlet heat equation on $\R_+$:
\begin{equation}
  \label{eq:outer-approx}
  r(t, x) \sim \rcoeff(a) \frac{x}{(t + 1)^{3/2}} e^{-x^2/4(t + 1)} \quad \trm{for } t \gg 1, \, x \asymp \sqrt{t}
\end{equation}
and some $\rcoeff(a) > 0$.

On the other hand, the asymptotics \eqref{eq:wave-limits} for $U$ extend naturally to $U'$.
In particular,
\begin{equation}
  \label{eq:wave-deriv-limit}
  U'(x - \bar{x}_0) \sim -\lambda_* x e^{-\lambda_*(x - \bar{x}_0)} \quad \trm{as } x \to \infty.
\end{equation}
Thus if $x$ is large but fixed, \eqref{eq:z-limit}, \eqref{20mar2508}, and \eqref{eq:wave-deriv-limit} suggest that
\begin{equation}
  \label{eq:inner-approx}
  r(t, x) \approx \lambda_* e^{\lambda_*(\bar{x}_0 + a)} \P(d_{12} > a) \frac{x}{(t + 1)^{3/2}} \quad \trm{as } t \to \infty.
\end{equation}
Comparing \eqref{eq:outer-approx} and \eqref{eq:inner-approx}, we expect
\begin{equation*}
  \rcoeff(a) = \lambda_* e^{\lambda_*(\bar{x}_0 + a)} \P(d_{12} > a).
\end{equation*}

Once again, the methods of \cite{HNRR,NRR1} confirm these calculations:
\begin{proposition}
  \label{prop:diffusive-scale}
  For each $a > 0$ and $\gamma \in \left(0,1/2\right)$, there exist $C(a, \gamma) > 0$ and $h \colon [0, \infty) \to \R$ such that for all $t \geq 0$ and $x \in \R$,
  \begin{equation}
    \label{20mar2520}
    r(t,x) = \left[\lambda_* e^{\lambda_*(\bar{x}_0 + a)} \P(d_{12} > a) + h(t)\right]\farc{x_+}{(t+1)^{3/2}} e^{-{x^2}/{4(t+1)}} + R(t, x)
  \end{equation}
  with
  \begin{equation}
    \label{eq:original-error}
    \abs{R(t, x)} \leq C(a, \gamma) (t + 1)^{-(3/2 - \gamma)} e^{-x^2/[6(t + 1)]} \And h(t) \to 0 \quad \trm{as } t \to \infty.
  \end{equation}
\end{proposition}
We have used the notation $x_+ = \max\{x, 0\}$ in \eqref{20mar2520}.
Again, this proposition is very similar to Proposition~4.1 and Lemma~5.1 in~\cite{NRR1}, so we omit the proof.

\subsection*{A moment computation}

By \eqref{eq:shift-deriv} and \eqref{20mar2520}, it now suffices to understand the long-time dynamics of $r$ on the diffusive scale.
We capture these dynamics in a certain moment of $r$.
An elementary computation using \eqref{20mar2520} and \eqref{eq:original-error} shows that
\begin{equation}
  \label{eq:moment-prelim}
  \P(d_{12} > a) = \frac{1}{2 \lambda_* \sqrt{\pi}} e^{-\lambda_*(a + \bar{x}_0)} \lim_{t \to \infty} \int_0^\infty x r(t, x) \d x.
\end{equation}
This identity is the reason for introducing the $(t + 1)^{-3/2}$ factor in \eqref{20mar2508}.
It is convenient to express this moment in terms of the function
\begin{equation*}
  \psi(x)=-U_0'(x) e^{\lambda_* x},
\end{equation*}
where we have introduced the notation
\begin{equation}
  \label{eq:front-shift}
  U_0(x) = U(x - \bar x_0)
\end{equation}
for the limiting shift of the front $U$.
The function $\psi$ is a positive time-independent solution to the leading-order part of (\ref{eq:main}):
\begin{equation}\label{mar1906bis}
0 = \psi'' - V(\infty,x) \psi = \psi''- \nonlin'(U_0(x))\psi.
\end{equation}
Here we have used \eqref{eq:potential-infty} and \eqref{eq:front-shift} to express the long-time limit of the potential $V$ introduced in \eqref{20may1214}.
Recalling that $\gamma_*+\lambda_*=\sqrt N$, (\ref{eq:wave-limits}) and an elementary ODE argument imply that
\begin{equation}
  \label{eq:adjoint-boundary}
  \begin{aligned}
    \psi(x) &=\Ufactor\gamma_* e^{-\gamma_*\bar x_0}e^{\sqrt{N}x} + O\left(e^{(\sqrt{N}+c)x}\right)&&\textrm{as } x\to-\infty\\
    \psi(x) &=\lambda_* e^{\lambda_*\bar x_0}x + O(e^{-c x})&&\textrm{as } x\to+\infty
  \end{aligned}
\end{equation}
for some $c>0$.
It follows from (\ref{20mar2520}), \eqref{eq:original-error}, and (\ref{eq:adjoint-boundary}) that \eqref{eq:moment-prelim} can be written as
\begin{equation}
  \label{20mar2520bis2}
  \P(d_{12} > a) = \frac{1}{2 \lambda_*^2 \sqrt{\pi}} e^{-\lambda_*(a + 2\bar{x}_0)} \lim_{t\to+\infty} I(t),
\end{equation}
where
\be\label{20mar2626}
I(t)=\int_\R \psi(x) r(t,x) \d x.
\end{equation}
This completes our series of reductions: to control $\P(d_{12} > a)$, we determine the dependence of $I(\infty)$ on $a$.

Motivated by \eqref{eq:potential-infty}, we write (\ref{eq:main}) as
\begin{equation}\label{mar1904}
\pdr{r}{t} = \pdrr{r}{x} -\farc{3}{2\lambda_*(t+1)}\pdr{r}{x} - V(\infty,x) r + E(t,x)r,
~~~r(0,x)=\delta(x+a), 
\end{equation}
with an error term
\begin{equation}\label{20mar2608}
E(t,x)= V(\infty,x) - V(t, x).
\end{equation}
If we multiply (\ref{mar1904}) by $\psi$ and integrate, \eqref{mar1906bis} yields
\begin{equation}\label{mar1912}
\der{I}{t}(t)=
\farc{3}{2\lambda_*(t+1)} \int_\R  {r(t,x)\psi'(x)} \d x +
\int_\R E(t,x) r(t,x)\psi(x) \d x.
\end{equation}

We will see below that the second term in the right side of (\ref{mar1912}) is, indeed, an error term, so we focus on the first term.
To this end, we describe the dynamics of (\ref{mar1904}) qualitatively, ignoring the error term $E r$.
First note that the mass of the solution on $\R_-$ will decay exponentially in time under \eqref{mar1904},
due to absorption from the term~$-\nonlin'(U_0)r$.
After all,~\eqref{eq:nonlin-deriv} implies~$\nonlin'(U_0)>0$ and, in particular, $\nonlin'(1)={N}$.  
However, mass that escapes to~$\R_+$ experiences almost no absorption because $\nonlin'(0)=0$. 
This ``fugitive'' mass diffuses, but gets absorbed whenever it returns to~$\R_-$.
Thus, as noted previously, \eqref{mar1904} acts much like the heat equation on $\Rm_+$ 
with a Dirichlet   boundary condition at $x=0$.
That said, there is initially no mass on $\R_+$, so we must include an initial time layer
during which the mass escapes from~$\Rm_-$ to~$\Rm_+$.

The initial condition $r(0,x)=\delta(x+a)$ is  ``deep in the large absorption territory,'' 
and takes a while to diffuse to $\R_+$. If we neglect the drift in the right side
of (\ref{mar1904}), then, 
in the absence of absorption, the proportion of mass that diffuses from position $-a$ to $\R_+$ at time $t$ is roughly $e^{-{a^2}/{(4t)}}$.
Attrition by the absorbing potential approximately reduces this by the factor $e^{-Nt}$.
Thus to leading order, the mass that escapes to $\R_+$ at time~$t$ is
\begin{equation*}
  \exp\left(-Nt - {a^2}/{(4t)}\right).
\end{equation*}
Therefore, the mass flux into $\R_+$ is maximized at the time 
\begin{equation}\label{20mar2630}
t_* = \farc{a}{2\sqrt{N}},
\end{equation}
and this is, roughly, the size of the initial time layer.  

Let us now use this heuristic to approximate the first term  
in the right side of (\ref{mar1912}):
\begin{equation*}
  \farc{3}{2\lambda_*(t+1)}\int_\R  {r(t,x)\psi'(x)} \d x = \farc{3}{2\lambda_*(t+1)} \int_\R \farc{ \psi'(x)}{\psi(x)}\psi(x)r(t,x) \d x.
\end{equation*}
We see from (\ref{eq:adjoint-boundary}) that
\begin{equation*}
\frac{\psi'(x)}{\psi(x)} \sim
\begin{cases}
\sqrt{N} & \trm{for } x\ll -1,\\
\disp\frac{1}{x} & \trm{for } x \gg 1.
\end{cases}
\end{equation*}
The heuristic argument above suggests that most of the mass lies in $\R_-$ 
until $t_*$, when it transfers to $\R_+$.
Therefore, we expect that
\begin{equation*}
\int_\R \frac{\psi'(x)}{\psi(x)}\psi(x)r(t,x) \d x \approx \sqrt{N} 
\int_\R  \psi(x)r(t,x) \d x
~~\hbox{ for $t < t_*$.}
\end{equation*}
After $t_*$, the mass of $r$ will largely stay in $\R_+$, and will spread to the diffusive scale $x \asymp \sqrt{t}$.
Since $\frac{\psi'}{\psi} \approx \frac{1}{x}$, we should have 
\begin{equation*}
\int_\R \frac{\psi'(x)}{\psi(x)}\psi(x)r(t,x) \d x \ll \int_\R  \psi(x)r(t,x) \d x
~~ \hbox{ for $t > t_*$.}
\end{equation*}
Thus, $I(t)$ satisfies
\begin{equation*}
\der{I}{t}(t)\approx \farc{3\sqrt{N}}{2\lambda_*(t+1)} I(t)~~\hbox{ for $t
< t_*$,}\qquad\qquad
\der{I}{t}(t)\approx 0~~\hbox{ for $t > t_*$. }
\end{equation*}
Integrating, we find
\begin{equation}\label{20mar2634}
\lim_{t\to+\infty}I(t) \approx t_*^{{3\sqrt{N}}/{(2\lambda_*)}} I(0)
=\left(\farc{a}{2\sqrt{N}}\right)^{{3\sqrt{N}}/{(2\lambda_*)}} I(0)
\end{equation}
when $a \gg 1$.
In the remainder of this paper, we justify~(\ref{20mar2634}).
\begin{proposition}\label{prop:main}
We have
\begin{equation}\label{eq:limit}
\lim_{t\to+\infty} I(t)=
\left(\farc{a}{2\sqrt{N}}\right)^{{3\sqrt{N}}/{(2\lambda_*)}} \left[1 + O\left(a^{-1/ 2}\right)\right] 
\int_\R \psi(x) r(0,x) \d x ~~~\hbox{ as $a\to+\infty$.}
\end{equation}
\end{proposition}
Theorem~\ref{thr:main} is a consequence of Proposition~\ref{prop:main}.
Indeed, using (\ref{eq:adjoint-boundary}), we see that
the integral in the right side of~(\ref{eq:limit})~is  
\begin{equation}\label{mar2602}
\begin{aligned}
\int_\R \psi(x) r(0,x) \d x =\psi(-a)=\Ufactor\gamma_* e^{-\gamma_*\bar x_0}e^{-\sqrt{N}a}\big(1+o(e^{-c a})\big)
\end{aligned}
\end{equation}
for some $c>0$.
We then use (\ref{20mar2520bis2}) to write
\begin{equation*}
  \begin{aligned}
    \P(d_{12} > a) &= \frac{1}{2 \lambda_*^2 \sqrt{\pi}} e^{-\lambda_*(a + 2\bar{x}_0)} \lim_{t\to+\infty} I(t)\\
&=\farc{\Ufactor\gamma_*}{2 \lambda_*^2 \sqrt{\pi}}e^{-(2\lambda_*+\gamma_*)\bar x_0}
\left(\farc{a}{2\sqrt{N}}\right)^{{3\sqrt{N}}/{(2\lambda_*)}} e^{-\big(\lambda_* + \sqrt{N}\big)a}
\left[1 + O\left(a^{-1/2}\right)\right] \big(1+o(e^{-c a})\big)\\
&=\farc{\Ufactor\gamma_*}{2 \lambda_*^2 \sqrt{\pi}}
\left(\farc{a}{2\sqrt{N}}\right)^{{3\sqrt{N}}/{(2\lambda_*)}}
e^{-\big(\sqrt{N}+\sqrt{N-1}\big)(a+\bar x_0)} \left[1 + O\left(a^{-1/2}\right)\right].
\end{aligned}
\end{equation*}
This finishes the proof of Theorem~\ref{thr:main}.
 
The rest of the paper contains the proof of Proposition~\ref{prop:main}.
The strategy is to estimate
the ratio~$\dot I(t)/I(t)$ on various times scales.
Lemmas~\ref{lem-may1506}, \ref{lem-may2202}, and \ref{20may2420} below express these estimates.
At the end of Section~\ref{sec:postmodern}, we collect these results and prove Proposition~\ref{prop:main}.
We note that Proposition~\ref{prop:main} likely holds for general Fisher--KPP nonlinearities, but the significance of the left side is far from clear when $f$ is not probabilistic.
In our proof, we confine our attention to probabilistic $f$.
Finally, because we are interested in the regime $a \to \infty$, we always implicitly assume that $a \geq 1$.
 
\section{The time scales and the correctors}
\label{sec:heuristics}

We now turn to the proof of Proposition~\ref{prop:main}. In this section, we discuss the time scales on which
various approximations to the dynamics of (\ref{eq:main}) should be valid, and introduce 
the corresponding ``scattering decomposition'' of the solution. 

\subsection*{The time scales}

Let us first explain the time
scales on which various effects will dominate. 
Our previous reasoning indicates that at times $t < t_*$, the heat equation has not 
had enough time to diffuse much mass from the initial position $x=-a$ to $x\ge 0$.
The evolution of $r(t,x)$, the solution to (\ref{eq:main}), is thus dominated by its homogeneous part
\begin{equation}\label{eq:freebis}
  \pdr{p}{t} = \pdrr{p}{x}-\frac 3 {2\lambda_*(t+1)} \pdr{p}{x} - Np, \quad p(0,x) = \delta(x+a),
\end{equation}
as $\nonlin'(1)=N$. 
Its explicit solution is
\begin{equation}\label{eq:explicit}
p(t,x) = \frac 1 {\sqrt{4\pi t}} \exp\Big\{-Nt - \frac{1}{4t}\Big[x + a - 
\frac{3}{2\lambda_*} \log(t+1)\Big]^2\Big\}.
\end{equation}
The corrector
\begin{equation*}
q(t,x) =r(t,x) - p(t,x)
\end{equation*}
solves
\begin{equation}\label{eq:refugebis}
\pdr{q}{t} = \pdrr{q}{x} - \frac{3}{2\lambda_*(t+1)} \pdr{q}{x} -V(t,x)q + 
(N-V(t,x)) p, \quad q(0,x) = 0,
\end{equation}
recalling that
\begin{equation*}
V(t,x)=\nonlin'(H(t,x+m(t)).
\end{equation*}
Since $h'(u) \leq N$ by \eqref{eq:nonlin-deriv}, $V \leq N$.
It follows that
\begin{equation}
  \label{eq:q-nonneg}
  q(t, x) \geq 0.
\end{equation}

We view \eqref{eq:main} as a perturbation of the absorbing heat equation (\ref{eq:freebis}),
so that $p(t,x)$ represents the free evolution and $q(t,x)$ accounts for the interaction 
with  the potential $V(t,x)$.
The role of $p$ is to transport the mass from $x=-a$ to the half-line $\R_+$, and the role of $q$ 
is to account for this escaped mass as $t \to \infty$.
Indeed, we will show that
\begin{equation*}
  I(t)=\int_\R r(t,x)\psi(x) \d x \approx \int_\R p(t,x)\psi(x)\d x \quad \textrm{for } t\ll  t_*
\end{equation*}
and
\begin{equation*}
  I(t)=\int_\R r(t,x)\psi(x) \d x \approx \int_\R q(t,x)\psi(x) \d x\quad \textrm{for } t\gg t_*.
\end{equation*}

In constructing our time scales, we must also consider the forcing term $(N - V)p$ in (\ref{eq:refugebis}).
In particular, we study its moment contribution
\begin{equation}\label{20mar2718}
\int_\R (N-V(t,x))\psi(x)  p(t,x) \d x.
\end{equation}
To understand the time scales on which it may potentially play a role, note that
the first two terms in the integrand decay on the left: 
$\psi(x)$ has the asymptotic behavior as in (\ref{eq:adjoint-boundary}),  
and the first factor  is controlled by the following lemma.
\begin{lemma}
  \label{lem:front-decay}
  There exist $B > 0$ and $c > 0$ depending only on $f$ such that
  \begin{equation}
    \label{20mar2720}
    0 \leq N - V(t,x)
    \leq \min\big\{B e^{\gamma_* x}, N\big\} 
    \And
    0 \leq V(t, x) \leq \min\big\{B e^{-cx}, N\big\}
  \end{equation}
  for all $t\ge 0$ and $x\in\Rm$.
\end{lemma}

\begin{proof}
For $t > 0$, define the shift $\ti{m}(t)$ by
\begin{equation*}
  H(t, \ti{m}(t)) = U_0(0).
\end{equation*}
This is well-defined and continuous because $H(t, \anon)$ strictly decreases from $1$ to $0$ and $H$ is continuous when $t > 0$.
By Theorem~12 in \cite{KPP},
\begin{equation}
  \label{eq:KPP-inequality-left}
  U_0(x) \leq H(t, x + \ti{m}(t)) \leq 1 \quad \textrm{for all } t > 0  \textrm{ and } x \leq 0
\end{equation}
and
\begin{equation}
  \label{eq:KPP-inequality-right}
  0 \leq H(t, x + \ti{m}(t)) \leq U_0(x) \quad \textrm{for all } t > 0  \textrm{ and } x \geq 0.
\end{equation}
Bramson's work \cite{Bramson1, Bramson2} implies that $\ti{m}(t) - m(t) \to 0$ as $t \to \infty$.
Using \eqref{eq:wave-limits} and \eqref{eq:KPP-inequality-left}, we find
\begin{equation}
  \label{eq:H-left}
  0 \leq 1 - H(t, x + m(t)) \leq \min\big\{C e^{\gamma_* x}, 1\big\} \quad \textrm{for all } t \geq 0 \textrm{ and } x \in \R
\end{equation}
for some $C$ depending only on $f$.
Similarly, \eqref{eq:wave-limits} and \eqref{eq:KPP-inequality-right} imply
\begin{equation}
  \label{eq:H-right}
  0 \leq H(t, x + m(t)) \leq \min\big\{C e^{-\lambda_* x/2}, 1\big\} \quad \textrm{for all } t \geq 0 \textrm{ and } x \in \R.
\end{equation}

Now, \eqref{20mar2208} and \eqref{eq:higher-moment} imply that $f$ is $\m C^2$ away from $0$ and $\m C^{1, \beta}$ near $0$.
The decomposition \eqref{20mar2434} implies the same regularity for $\nonlin$.
Using the definition \eqref{20may1214} of $V$, the first bound in \eqref{20mar2720} follows from \eqref{eq:H-left}, \eqref{eq:nonlin-deriv}, and $\nonlin' \in \m{C}^1([1/2, 1])$.
The second bound in \eqref{20mar2720} follows from \eqref{eq:H-right}, \eqref{eq:nonlin-deriv}, and $\nonlin' \in \m{C}^\beta([0, 1])$.
\end{proof}

To understand the spatial decay of $p(t,x)$, it helps to write it as
\begin{equation}\label{eq:p-zero}
p(t, x) = \Lambda(t;a)e^{-{ax}/{(2t)}} g(t, x)
\end{equation}
with the factor
\begin{equation}\label{eq:Lambda}
\Lambda(t;a) = \frac{(t+1)^{{3a}/{(4\lambda_* t)}}}{\sqrt{4 \pi t}} e^{-Nt -{a^2}/{(4t)}}
\end{equation}
and the re-centered Gaussian
\begin{equation}\label{eq:g}
g(t, x)=\exp\Big\{-\frac{1}{4t}\Big[x - \frac{3}{2\lambda_*}\log(t+1)\Big]^2\Big\}
\end{equation}

Combining (\ref{eq:adjoint-boundary}), (\ref{20mar2720}), and \eqref{eq:p-zero},
we see that it is straightforward to control the spatial decay
of the integrand in (\ref{20mar2718}) as $x\to-\infty$ for times $t$ such that
\begin{equation*}
\gamma_*+\sqrt{N}>\farc{a}{2t},
\end{equation*}
i.e.
\begin{equation*}
t>\farc{a}{2(2\sqrt{N}-\sqrt{N-1})}. 
\end{equation*}
Accordingly, we fix 
\begin{equation*}
\earlyxi\in\Big(\farc{1}{2(2\sqrt{N}-\sqrt{N-1})}, \, \farc{1}{2\sqrt{N}}\Big), \qquad\ \earlyt=\earlyxi a.
\end{equation*}

Let us decompose the  solution to \eqref{eq:refugebis}  
as 
\begin{equation*}
q=\qe+\qm,
\end{equation*}
where $\qe$ is forced on the time interval $[0, \earlyt]$ and $\qm$ on $[\earlyt, \infty)$:
\begin{equation}
  \label{eq:scattering-early}
  \pdr{\qe}{t} = \pdrr{\qe}{x} - \frac{3}{2\lambda_*(t+1)} \pdr{\qe}{x} 
  - V(t,x) \qe + (N-V(t,x)) \one_{[0, \earlyt]}(t)p
\end{equation}
and
\begin{equation}
  \label{eq:scattering-main}
  \pdr{\qm}{t} = \pdrr{\qm}{x} - \frac{3}{2\lambda_*(t+1)} \pdr{\qm}{x}
  - V(t,x) \qm + (N-V(t,x)) \one_{[\earlyt, \infty)}(t)p,
\end{equation}
with $\qe(0, \anon) = \qm(0, \anon) = 0$.

As we have discussed,  the product $(N-V)p$ is very small for $t\ll t_*$, and $\earlyt \ll t_*$ if~$a\gg 1$.
It follows that $\qe$ should never form a significant part of $r$, and we think of it as \emph{error}.
In contrast, $\qm$ is eventually the principal part of $r$, so we view it as the \emph{main} part of $q$.

Although $\qe$ should be irrelevant, it is challenging to estimate.
We introduced the corrector $q$ because it is easier to analyze the long-time behavior of adjoint-weighted mass which begins in $\R_+$, rather than deep in $\R_-$.
As we argue above, the adjoint-weighted forcing $(N - V) \psi p$ for $\qm$ is concentrated on $\R_+$.
However, this is not the case for $\qe$, which is still primarily forced deep in $\R_-$.
We appear to be no better off than when we started with a point mass at $-a$!
And indeed, we will be forced to control $\qe$ using a further corrector.

However, $\qe$ is driven by the forcing $(N - V)p$, which is far smaller than $p$ due to Lemma~\ref{lem:front-decay}.
It follows that $\qe$ is much smaller than the original solution $r$.
We can therefore be less precise in our estimation of $\qe$.
This wiggle room saves us from a futile infinite descent of correctors.
Instead, two steps suffice.
The details are rather technical, so we defer them to Section \ref{sec:qe}.
There, we show:
\begin{lemma}
  \label{lem:early-corrector}
  There exist $C, c > 0$ independent of $a$ and $t$ such that for all $a \geq 1$ and $t \geq 0$:
  \begin{equation*}
    \begin{split}
      &\int_\R \psi(x) \qe(t,x) \d x \leq Ce^{-(\sqrt{N}+c)a},\\
      &\int_\R \psi'(x) \qe(t,x) \d x \leq \frac{C}{\sqrt{t + 1}} e^{-(\sqrt{N}+c)a},\\
      &\int_\R  \psi(x) \abs{E(t,x)} \qe(t,x) \d x \leq \farc{C}{(t + 1)^2} e^{-(\sqrt{N}+c)a}.
    \end{split}
  \end{equation*}
\end{lemma}
(We note that $\psi$ and $\psi'$ are non-negative, so the first two bounds are effective.)
These bounds feature an extra factor of $e^{-ca}$ relative to the main term, which is of order $e^{-\sqrt{N}a}$; \emph{Cf.} \eqref{mar2602}.
This justifies our treatment of $\qe$ as error.

To prove Lemma~\ref{lem:early-corrector}, we will make use of the following
result, which will also be useful in subsequent sections:
\begin{lemma}
  \label{lem:super}
  Let $w (t,x)$ satisfy
  \begin{align}
    \label{eq:super}
    &\pdr{w}{t} \leq \pdrr{w}{x} - \frac{3}{2 \lambda_* (t + 1)} \pdr{w}{x} - \al^2
      \one_{(-\infty,-K)}(x) w ,~~t>s,~x\in\Rm,\\
    \label{eq:super-initial}
    &w(s, x) \leq
      \begin{cases}
        e^{-\kappa_- x} & \textrm{for } x < 0,\\
        e^{-\kappa_+ x}e^{-x^2/(8s)} & \textrm{for } x \geq 0,
      \end{cases}
  \end{align}
  for some $\al > 0$, $K \geq 0$, $\kappa_- < \al$, and $\kappa_+ > 0$.
  Then there exists a constant $C > 0$ which depends on $\al, K$, and $\kappa_-$ but \emph{not} on $s$ or $\kappa_+$ such that for all $t \geq s$, we have
  \begin{align}
    \label{eq:psi-super}
    & \int_\R \psi(x) w(t, x) \d x\leq C \max\{\kappa_+^{-2}, 1\},\\
    \label{eq:psi-deriv-super}
    &\int_\R \psi'(x) w(t, x) \d x\leq C \max\{\kappa_+^{-2}, 1\} (t - s + 1)^{-1/ 2},\\
    \label{eq:wave-super}
    &\int_\R |E(t,x)| \psi(x) w(t, x) \d x\leq C \max\{\kappa_+^{-2}, 1\} 
      (t + 1)^{-1/ 2} (t - s + 1)^{-{3}/{2}}.
  \end{align}
\end{lemma}

We prove Lemma~\ref{lem:super} in Appendix~\ref{sec:appendix}, but offer a heuristic explanation here.
Roughly speaking, we may think of \eqref{eq:super} as the heat equation on 
the half-line with Dirichlet boundary conditions, so that
\begin{equation}\label{20may1524}
w(t, x) \sim  \frac{x}{(t - s + 1)^{3/2}}e^{-{x^2}/{(4(t - s + 1))}}
\int_0^\infty xw(s,x)\d x\le \farc{C}{\kappa_+^2}
\frac{x}{(t - s + 1)^{3/2}}e^{-{x^2}/{(4(t - s + 1))}}.
\end{equation}
The bounds in \eqref{eq:psi-super}--\eqref{eq:wave-super} come from the right side
of (\ref{20may1524}). 

Going forward, we separately consider three regimes delimited by $t_*$ from \eqref{20mar2630} and $a$:
the early times $0\le t\le  t_*$ in Section~\ref{sec:early-modern}, the middle times  $t_* \leq t \leq a$ in Section~\ref{sec:late-modern}, and the late times $t\ge a$ in Section~\ref{sec:postmodern}.
The cutoff $a$ is more or less arbitrary: it simply allows us to assume that $t = O(a)$ in the middle regime.
As stated above, Proposition~\ref{prop:main} and
Theorem~\ref{thr:main} are immediate consequences of 
Lemmas~\ref{lem-may1506}, \ref{lem-may2202}, and~\ref{20may2420}.

\section{The early times}
\label{sec:early-modern}
In this section, we start analyzing the  contributions of $p$ and
$q=\qm+\qe$ 
to $I(t)$ and $\dot I(t)$.
We begin with the early times $t \leq t_*$.
The terms involving the ``early'' corrector $\qe$ have already
been bounded in Lemma~\ref{lem:early-corrector} and will turn out to be
irrelevant---they are much smaller than the corresponding  contributions of
$p(t,x)$ in Lemma~\ref{lem-may1902} below. It  remains to estimate the terms 
involving $p(t,x)$ and $\qm(t,x)$. We will see that the terms involving $p$
dominate for nearly the entire interval $t\in[0,t_*]$ (in particular,
$\qm\equiv0$ for $t\le\earlyt$). However, the contributions of $p$ and
$\qm$ become comparable when $t_* - t = O\big(\sqrt{a}\big)$.

The main result of this section is the following lemma.

\begin{lemma}\label{lem-may1506}
We have 
\begin{equation}\label{20may1538}
\frac{\dot I(t)}{I(t) }= 
\frac{3\sqrt{N}}{2\lambda_*(t + 1)} 
\left[1 + O\left(e^{-\frac{c( t_*-t)^2}{a}}\right)\right]
\end{equation}
for all $t\in[0, t_*]$.
\end{lemma}
We note that the error term in \eqref{20may1538} is small over most of the time period $[t, t_*],$ but it becomes order $1$ when $t_* - t = O\big(\sqrt{a}\big)$.
After integration, Lemma~\ref{lem-may1506} implies that
\begin{equation}\label{eq-early-int}
\log\frac{I(t_*)}{I(0)}=\frac{3\sqrt{N}}{2\lambda_*}\log(t_*+1)+ O(a^{-1/2})
=\frac{3\sqrt{N}}{2\lambda_*}\log\frac{a}{2\sqrt N}+ O(a^{-1/2}),
\end{equation}
where the error term is dominated by the region $t_* - t =             O\big(\sqrt{a}\big)$.

\subsection*{The free contribution}

We first look at the putative main term in this period.
\begin{lemma}
  \label{lem-may1902}
  There exist $c>0$ and $C> 0 $ so that for all $t \in [\earlyt,t_*]$, we have the asymptotics
\begin{equation}\label{20may1920}
  \int_\R \psi(x) p(t,x) \d x = \Ufactor \gamma_*e^{-\gamma_*\bar x_0}(t + 1)^{{3\sqrt{N}}/{(2\lambda_*)}} 
e^{-\sqrt{N}a}\Big[1 + O\Big(e^{-\frac{c(t_*-t)^2}{a}}\Big)\Big]
\end{equation}
and
\begin{equation}\label{20may1922}
\int_\R \psi'(x)p(t, x) \d x = \sqrt{N}\Ufactor \gamma_*e^{-\gamma_*\bar x_0}(t + 1)^{{3\sqrt{N}}/{(2\lambda_*)}} 
e^{-\sqrt{N}a} \Big[1 + O\Big(e^{-\frac{c(t_*-t)^2}{a}}\Big)\Big],
\end{equation}
as well as the error estimate
\be\label{20may1924}
\int_\R|E(t,x)| p(t,x) \psi(x)\d x\le \frac C {t+1} e^{-\frac{c(t_*-t)^2}{a}}
 \int_\R \psi'(x) p(t,x)\d x.
\ee
Moreover, for all $t \in [\earlyt, t_*]$,
\begin{equation}\label{20may1925}
I(t) \geq \int_\R \psi(x) p(t,x) \d x \ge c (t + 1)^{{3\sqrt{N}}/{(2\lambda_*)}} e^{-\sqrt{N}a}.
\end{equation}
\end{lemma}
The first estimate in this lemma permits $\int_\R \psi p$ to reach $0$ when $t_*-t = O\big(\sqrt{a}\big)$, which could make $\dot I(t)/I(t)$ very large.
The lower bound (\ref{20may1925}) ensures that this does not occur.

\begin{proof}
To prove (\ref{20may1920}), let us first re-write $p(t,x)$ in a more
convenient form, starting from \eqref{eq:explicit}:
\begin{equation}
\label{20may1404}
\bal
p(t,x)
& = \frac 1 {\sqrt{4\pi t}} \exp\Big\{-Nt - \frac{1}{4t}\Big[x + a - 
\frac{3}{2\lambda_*} \log(t+1)\Big]^2\Big\} 
\\
& = \frac 1 {\sqrt{4\pi t}} \exp\Big\{ - \frac{1}{4t}\Big[x + a - 
\frac{3}{2\lambda_*} \log(t+1)-2\sqrt N t\Big]^2
-\sqrt N\Big[x+a-\frac3{2\lambda_*}\log(t+1)\Big]\Big\} 
\\
&=\farc{1}{\sqrt{4\pi t}} {e^{-a\sqrt{N}}}(t+1)^{3\sqrt{N}/(2\lambda_*)}
 e^{-\sqrt{N}x}\exp\Big\{- \frac{1}{4t}\Big[x - \wshift(t)\Big]^2\Big\}.
 \enbal
\end{equation}
where we used the notation
\begin{equation}\label{eq:mu}
\wshift(t) = -a + \frac{3}{2\lambda_*} \log(t + 1) + 2\sqrt{N}t.
\end{equation}
It follows that
\begin{equation}
  \label{20may1928}
  \begin{aligned}
    &\int_\R \psi(x) p(t,x) \d x= {e^{-\sqrt{N}a}}(t+1)^{3\sqrt{N}/(2\lambda_*)}
    \int_\R\zeta(x)\exp\Big\{- \frac{1}{4t}\Big[x -\wshift(t)\Big]^2\Big\}\farc{\dn x}{\sqrt{4\pi t}},
  \end{aligned}
\end{equation}
with
\begin{equation*}
\zeta(x)=\psi(x)\exp\big(-\sqrt{N}x\big).
\end{equation*}
Note that  (\ref{eq:adjoint-boundary}) gives 
\begin{equation}
  \label{eq:zeta-indicator}
  \zeta(x) = \Ufactor \gamma_* e^{-\gamma_* \bar{x}_0} \one_{\R_-}(x) + O\big(e^{-c\abs{x}}\big).
\end{equation}
Note also that  
\be\label{20may1931}
\wshift(t)=-a+\farc{3}{2\lambda_*}\log (t+1)+2\sqrt{N}(t_*-(t_*-t))\le -2\sqrt{N}(t_*-t)+C\log a.
\ee
When $t_* - t \leq \sqrt{a}$, \eqref{20may1920} follows simply from $0 \leq \zeta \leq C$, which is a consequence of \eqref{eq:zeta-indicator}
We may thus assume that $0 \leq t \leq t_* - \sqrt{a}$.
Since $\log a \ll \sqrt{a}$ when $a$ is large, we can in turn assume that \eqref{20may1931} yields
\begin{equation}
  \label{eq:nu-simple}
  \nu(t) \leq -c(t_* - t).
\end{equation}
Using \eqref{eq:zeta-indicator} and \eqref{20may1931}, we can write the integral in the right side of \eqref{20may1928} as
\begin{equation}
  \begin{aligned}
    \label{eq:zeta-p-condensed}
    \int_\R\zeta(x)\exp\Big\{- \frac{1}{4t}\Big[x -\wshift(t)\Big]^2\Big\}\farc{\dn x}{\sqrt{4\pi t}} &= \Ufactor\gamma_* e^{-\gamma_*\bar x_0} \int_{\R_-} \exp\Big\{- \frac{1}{4t}\Big[x -\wshift(t)\Big]^2\Big\}\farc{\dn x}{\sqrt{4\pi t}}\\
    &\hspace{1cm}+O\left(\int_\R e^{-c\abs{x}}\exp\Big\{- \frac{1}{4t}\Big[x -\wshift(t)\Big]^2\Big\}\farc{\dn x}{\sqrt{4\pi t}}\right).
  \end{aligned}
\end{equation}
We can write the main term on the right side as
\begin{equation}
  \label{eq:zeta-p-main}
  \begin{aligned}
    \int_{\R_-} \exp\Big\{- \frac{1}{4t}\Big[x -\wshift(t)\Big]^2\Big\}\farc{\dn x}{\sqrt{4\pi t}} &= 1 - \int_{\R_+} \exp\Big\{- \frac{1}{4t}\Big[x + \abs{\wshift(t)}\Big]^2\Big\}\farc{\dn x}{\sqrt{4\pi t}}\\
    &= 1 + O\left(e^{-\nu(t)^2/(4t)}\right) =  1 + O\left(e^{-\frac{c(t_*-t)^2}{a}}\right).
  \end{aligned}
\end{equation}
We used \eqref{eq:nu-simple} and $t \leq a$ in the last step.
To bound the error in \eqref{eq:zeta-p-condensed}, we break the integral at the position $\nu(t)/2$.
For $x\le\nu(t)/2$ we
write, on the one hand,
\begin{equation*}
  \int_{-\infty}^{\nu(t)/2} e^{-c\abs{x}}\exp\Big\{- \frac{1}{4t}\Big[x -                     
\wshift(t)\Big]^2\Big\}\farc{\dn x}{\sqrt{4\pi t}} \leq 
\frac1{\sqrt{4\pi t}} \int_{-\infty}^{\nu(t)/2} e^{cx}\d x = \frac{C}{\sqrt
t} e^{c\nu(t)/2}
\end{equation*}
and, on the other hand
\begin{equation*}
  \int_{-\infty}^{\nu(t)/2} e^{-c\abs{x}}\exp\Big\{- \frac{1}{4t}\Big[x -
\wshift(t)\Big]^2\Big\}\farc{\dn x}{\sqrt{4\pi t}} \leq
e^{c\nu(t)/2}\int_\R\exp\Big\{- \frac{1}{4t}\Big[x -
\wshift(t)\Big]^2\Big\}\farc{\dn x}{\sqrt{4\pi t}} = e^{c\nu(t)/2}
\end{equation*}
Combining these two bounds gives
\begin{equation}
  \label{eq:zeta-p-error-left}
  \begin{aligned}
    \int_{-\infty}^{\nu(t)/2} e^{-c\abs{x}}\exp\Big\{- \frac{1}{4t}\Big[x - \wshift(t)\Big]^2\Big\}\farc{\dn x}{\sqrt{4\pi t}}
    \leq \frac{C}{\sqrt{1+t}} e^{c\nu(t)/2}
    &\leq \frac{C}{\sqrt{1+t}} e^{-c(t_* - t)}\\
    &\leq \frac{C}{\sqrt{1+t}} e^{-\frac{c(t_* - t)^2}{a}}.
  \end{aligned}
\end{equation}
For the the region $x\ge\nu(t)/2$ of the integral in the error term in
\eqref{eq:zeta-p-condensed}, we have
\begin{equation}
  \label{eq:zeta-p-error-right}
  \int_{\nu(t)/2}^\infty e^{-c\abs{x}}\exp\Big\{- \frac{1}{4t}\Big[x
-                        \wshift(t)\Big]^2\Big\}\farc{\dn x}{\sqrt{4\pi t}}
\leq \frac{C}{\sqrt{t}}e^{-\nu(t)^2/      (16t)} \leq \frac{C}{\sqrt{1+t}} e^{-\frac{c(t_* - t)^2}{a}}.
\end{equation}
In the last step, we wrote $\nu(t)^2/(16t)\ge 1/(32t)+\nu(t)^2/(32t)$ and used
\begin{equation*}
  \frac{1}{\sqrt{t}}\exp\left(-\frac{1}{32t}\right) \le \frac{C}{\sqrt{t + 1}}.
\end{equation*}

Now, \eqref{20may1920} follows from \eqref{20may1928},
\eqref{eq:zeta-p-condensed},
\eqref{eq:zeta-p-main},
\eqref{eq:zeta-p-error-left} and
\eqref{eq:zeta-p-error-right}.
The proof of \eqref{20may1922} is identical: we just need to
use
 $\tilde\zeta(x):=\psi'(x)\exp(-\sqrt N x)=
 \sqrt N \Ufactor \gamma_* e^{-\gamma_* \bar{x}_0} \one_{\R_-}(x) + O\big(e^{-              c\abs{x}}\big)
$
in place of
\eqref{eq:zeta-indicator}.

We now turn to \eqref{20may1924}.  Recalling the definitions
\eqref{20mar2608} of $E(t,x)$ and \eqref{20may1214} of $V(t,x)$, and
noting that $h'$ is Lipschitz, we recall the following bound from \cite{Graham,NRR1,NRR2}:
\be\label{20may1932}
|E(t,x)|\le C \big|U_0(x) - H(t, x + m(t))\big| \leq \farc{Ce^{-c|x|}}{\sqrt{t+1}}.
\ee
Strictly speaking, \cite{NRR2} only proves this bound for $x \geq 0.$
However, the wave $U_0$ is linearly stable on $\R_-$ and \eqref{eq:wave-limits} and \eqref{eq:KPP-inequality-left} imply that $U_0 - H$ decays exponentially in space.
Because
\begin{equation*}
  \abs{U_0(0) - H(t, m(t))} \leq C (t + 1)^{-1/2}
\end{equation*}
at the boundary $x = 0$ of $\R_-$, it is straightforward to control $\abs{U_0 - H}$ by a supersolution of the form $C (t + 1)^{-1/2} e^{-c\abs{x}}$ on $\R_-$, perhaps after translation.
We omit the details.

Inserting this factor in \eqref{20may1928} and recalling that $\zeta(x)\le
C$, we find
\begin{equation*}
  \int_\R \abs{E(t, x)} \psi(x) p(t, x) \d x \leq \frac{C}{\sqrt{t+1}} e^{-\sqrt{N}a} (t + 1)^{3\sqrt{N}/(2\lambda_*)} \int_\R e^{-c\abs{x}}  \exp\Big\{- \frac{1}{4t}\Big[x + \abs{\wshift(t)}\Big]^2\Big\}\farc{\dn x}{\sqrt{4\pi t}}.
\end{equation*}
We estimated this integral integral above, assuming $t_* - t \geq \sqrt{a}$.
Using \eqref{eq:zeta-p-error-left} and \eqref{eq:zeta-p-error-right}, we obtain \eqref{20may1924} for such times.
When $t_* - t \leq \sqrt{a}$, we can bound the Gaussian in the above integral by $1$ to obtain \eqref{20may1924}.

To obtain the lower bound in (\ref{20may1925}), we first recall \eqref{eq:q-nonneg}, which implies $r \geq p$ and $I \geq \int_\R \psi p$.
Next, we note that $\zeta \geq c \one_{\R_-}$.
So, \eqref{20may1928} implies
\begin{equation}
  \label{eq:zeta-p-lower-condensed}
  I(t) \geq \int_\R \psi(x) p(t,x) \d x \geq c (t + 1)^{3\sqrt{N}/(2\lambda_*)} e^{-\sqrt{N}a} \int_{\R_-} \exp\left\{-\frac{1}{4t}\left[x - \nu(t)\right]^2\right\} \frac{\dn x}{\sqrt{4 \pi t}}.
\end{equation}
Changing variables via $\eta = \frac{x - \nu(t)}{\sqrt{t}}$, we have
\begin{equation}
  \label{eq:partial-Gaussian}
  \int_{\R_-} \exp\left\{-\frac{1}{4t}\left[x - \nu(t)\right]^2\right\} \frac{\dn x}{\sqrt{4 \pi t}} = \int_{-\infty}^{-\nu(t)/\sqrt{t}} e^{-\eta^2/4} \frac{\dn \eta}{\sqrt{4 \pi}}.
\end{equation}
By \eqref{20may1931}, for $t\ge\earlyt$,
\begin{equation*}
  -\frac{\nu(t)}{\sqrt{t}} \geq -\frac{C \log a}{\sqrt{a}} \geq -1,
\end{equation*}
assuming $a$ is large.
In light of \eqref{eq:partial-Gaussian}, \eqref{eq:zeta-p-lower-condensed} yields
\begin{equation*}
  I(t) \geq \int_\R \psi(x) p(t,x) \d x \geq c (t + 1)^{3\sqrt{N}/(2\lambda_*)} e^{-\sqrt{N}a}.
\end{equation*}
This finishes the proof of Lemma~\ref{lem-may1902}.
\end{proof}

\subsection*{The main corrector  contribution}

Next, we control the contributions of $\qm$ to $I(t)$ and $\dot{I}(t)$.
We only need to consider $t\ge\earlyt$ as $\qm\equiv0$ for $t<\earlyt$.
\begin{lemma}
  \label{lem:qm-early-middle}
  There exist $c > 0$ and $C > 0$ such that for all $t \in [\earlyt, t_*]$,
  \begin{equation}
    \label{eq:qm-psi-psi'-early-middle}
    \int_\R [\psi(x) + \psi'(x)] \qm(t, x) \d x \leq C (t
+ 1)^{3\sqrt{N}/(2\lambda_*)} e^{-\sqrt{N} a} e^{-\frac{c(t_*-t)^2}{a}}
  \end{equation}
  and
    \begin{equation}
    \label{eq:qm-E-early-middle}
    \int_\R \abs{E(t, x)} \psi(x)\qm(t, x) \d x \leq \frac{C}{a} (t
+ 1)^{3\sqrt{N}/(2\lambda_*)} e^{-\sqrt{N} a} e^{-\frac{c(t_*-t)^2}{a}}.
  \end{equation}
\end{lemma}

Recall that~$\qm$ satisfies (\ref{eq:scattering-main}):
\begin{equation*}
  \pdr{\qm}{t} = \pdrr{\qm}{x} - \frac{3}{2\lambda_*(t+1)} \pdr{\qm}{x}
  - V(t,x) \qm + (N-V(t,x)) \one_{[\earlyt, \infty)}(t)p, \quad \qm(0, \anon) = 0.
\end{equation*}
We use the Duhamel formula
\be\label{20may2216}
\qm(t,x)=\int_{\earlyt}^t q^s(t,x) \d s
\ee
to control $\qm$.
Here, for $s \geq \earlyt$, the function $q^s$ satisfies
\begin{equation}\label{eq:q-s}
\bal
&\pdr{q^s}{t} = \pdrr{q^s}{x} - \frac{3}{2\lambda_*(t + 1)} \pdr{q^s}{x} - V(t,x)q^s,~~t>s,\\
& q^s(s, x) = \big[N-V(s, x)\big] p(s, x).
\enbal
\end{equation}
We can control this for all $\earlyt \leq s \leq t$.
\begin{lemma}
  \label{lem:qs-integrals}
  For all $s \ge \earlyt$ and $t \geq s$,
  \begin{equation}
    \label{eq:q-s-bounds-middle}
    \begin{aligned}
      &\int_\R \psi(x)q^s(t,x) \d x \leq C \frac{s^2}{a^2}\Lambda(s;a),\\
      &\int_\R \psi'(x)q^s(t,x) \d x \leq C\frac{s^2}{a^2}  \Lambda(s;a) (t - s + 1)^{-1/2},\\
      &\int_\R \abs{E(t,x)} \psi(x)q^s(t,x) \d x \leq C \frac{s^2}{a^2} \Lambda(s;a) (t + 1)^{-1/2} (t - s + 1)^{-3/2}.
    \end{aligned}
  \end{equation}
\end{lemma}
\begin{proof}
  We first get rid of the logarithmic term in \eqref{eq:p-zero} and \eqref{eq:g}.
  We claim that
  \begin{equation*}
    \bal
    p(s,x) &= \Lambda(s;a) e^{-ax/(2s)} \exp\Big[-\farc{1}{4s}\Big(x-\farc{3}{2\lambda_*}\log(s+1)\Big)^2\Big]\\
    &\leq C \Lambda(s;a) e^{-ax/(2s)} \exp\Big(-\farc{x^2}{8s}\Big).
    \enbal
  \end{equation*}
After all, an elementary study of the quadratic polynomial shows that for all $\al > 0$ and $\eps \in (0, 1)$, there there exists $C$ depending on $\al$ and $\eps$ such that
\be
\label{eq:gaussian mean bound}
\frac{[z+\alpha \log(t+1)]^2}{4t}\ge  (1-\eps) \frac{z^2}{4t}-C
\qquad\text{for all $z\in\R$ and $t>0$}.
\ee

By Lemma~\ref{lem:front-decay}, we obtain
\begin{equation}
  \label{eq:forcing-bound-middle}
  \Lambda(s; a)^{-1}q^s(s, x) = \Lambda(s; a)^{-1}[N - V(s, x)] p(s, x) \leq
  \begin{cases}
    C\exp\left[-\left(\frac{a}{2s} - \gamma_*\right)x\right] & \textrm{for } x < 0,\\
    C\exp\left(-\frac{ax}{2s}\right)\exp\Big(-\farc{x^2}{8s}\Big) & \textrm{for } x \geq 0.
  \end{cases}
\end{equation}
Since $s \ge\earlyt$, we have
\begin{equation}
  \label{eq:exponents}
  \frac{a}{2s} - \gamma_* \leq \frac{1}{2\earlyxi} - \gamma_* \eqqcolon
\kappa_- \in \big(0, \sqrt{N}\big).
\end{equation}
Let us fix $\al \in \big(\kappa_-, \sqrt{N}\big)$.
By Lemma~\ref{lem:front-decay}, there exists $K \geq 0$ such that
\begin{equation}
  \label{eq:front-big}
  V(t,x) \geq \al^2 \one_{(-\infty, -K)}(x).
\end{equation}
We see from \eqref{eq:forcing-bound-middle} and \eqref{eq:front-big} that
${q^s}/[{C \Lambda(s; a)}]$
satisfies the hypotheses of Lemma~\ref{lem:super}, with $\kappa_-$
given in \eqref{eq:exponents} and $\kappa_+=a/(2s)$.
The bounds in \eqref{eq:q-s-bounds-middle} follow from Lemma~\ref{lem:super}
and $s\ge\earlyt$.
\end{proof}

We can now control $\qm$ in terms of $q^s$.

\begin{proof}[Proof of Lemma~\ref{lem:qm-early-middle}]
In light of \eqref{20may2216}, we need to integrate (\ref{eq:q-s-bounds-middle}) over $s \in [\earlyt, t]$ for $t \in [\earlyt, t_*]$.
For $s\ge \earlyt$, the definition \eqref{eq:Lambda} of $\Lambda$ implies
\begin{equation*}
\Lambda(s;a) \le \farc{C}{\sqrt{a}} {(s+1)^{{3a}/{(4\lambda_* s)}}} e^{-a\theta(s/a)}
\end{equation*}
for the rate function
\begin{equation}
  \label{eq:rate-function}
  \theta(\xi) = N \xi + \frac{1}{4\xi} \quad \textrm{for } \xi > 0.
\end{equation}
This strictly convex function is minimized at $\xi_* = 1/(2 \sqrt{N})$, so there exists $c > 0$ such that
\begin{equation*}
\theta(\xi)=N\xi+\farc{1}{4\xi}\ge \theta(\xi_*)+c(\xi-\xi_*)^2 =\sqrt{N}+c(\xi-\xi_*)^2
\end{equation*}
for all $\xi \in (0, 1]$.
Therefore,
\be
\label{20may2102}
\Lambda(s;a) \le \farc{C}{\sqrt{a}}  {(s+1)^{{3a}/{(4\lambda_* s)}}}
e^{-\sqrt{N}a -\frac{c( t_*-s)^2}{a}}
\ee
for all $s \in [\earlyt, a]$.

We next handle the polynomial prefactor, which we write as
\begin{equation*}
  (s+1)^{{3a}/{(4\lambda_* s)}} = \exp\left[\frac{3 a}{4 \lambda_* s} \log(s + 1)\right].
\end{equation*}
For any $\eps > 0$, we employ the Peter--Paul inequality $2AB\le\eps A^2+B^2/\eps$:
\begin{equation*}
  a \left(\frac{1}{s} - \frac{1}{t_*}\right)\log(s + 1) \leq \frac{C(t_* - s)\log a}{a} \leq \frac{\eps(t_* - s)^2}{a} + \frac{C^2\log^2 a}{4 \eps a} \leq \frac{\eps(t_* - s)^2}{a} + C_\eps.
\end{equation*}
Exponentiating, this implies that
\begin{equation}
  \label{eq:fix-polynomial-exponent}
  (s+1)^{{3a}/{(4\lambda_* s)}} \leq C_\eps (s + 1)^{3a/(4 \lambda_* t_*)}
e^{\frac{\eps (t_*-s)^2}{a}} = C_\eps (s + 1)^{3\sqrt{N}/(2 \lambda_*)}
e^{\frac{\eps (t_*-s)^2}{a}}.
\end{equation}
Taking $\eps \ll 1$, we can absorb the last factor into the Gaussian term in \eqref{20may2102}.
Thus, \eqref{20may2102} and \eqref{eq:fix-polynomial-exponent} yield
\begin{equation}
  \label{20may2108}
  \Lambda(s;a) \le \frac{C}{\sqrt{a}} (s+1)^{3\sqrt{N}/(2\lambda_*)}
e^{-\sqrt{N}a - \frac{c(t_*-s)^2}{a}} \leq \frac{C}{\sqrt{a}} (s
+ 1)^{3\sqrt{N}/(2\lambda_*)} e^{-\sqrt{N}a - \frac{c(t_*-s)^2}{a}}
\end{equation}
for all $s \in [\earlyt, a]$.
We now combine (\ref{eq:q-s-bounds-middle}) and (\ref{20may2108}) to control the contribution of $\qm$.
At these times, we do not need the distinction between $\psi$ and $\psi'$ in \eqref{eq:q-s-bounds-middle}:
\begin{equation*}
  \int_\R [\psi(x)+\psi'(x)]\qm(t,x) \d x=\int_{\earlyt}^t \int_\R [\psi(x) + \psi'(x)]q^s(t,x) \d x \ds s \leq C \int_{\earlyt}^t \Lambda(s;a) \d s.
\end{equation*}
Using \eqref{20may2108} and changing variables via $\eta = \frac{t_* - s}{\sqrt{a}}$, we obtain
\begin{equation*}
  \begin{aligned}
    \int_\R [\psi(x)+\psi'(x)]\qm(t,x) \d x &\leq C (t + 1)^{3\sqrt{N}/(2\lambda_*)} e^{-\sqrt{N}a} \int_{(t_* - t)/\sqrt{a}}^\infty e^{-c\eta^2} \d \eta\\
    &\leq C (t + 1)^{3\sqrt{N}/(2\lambda_*)} e^{-\sqrt{N}a}
e^{-\frac{c(t_*-t)^2}{a}}.
\end{aligned}
\end{equation*}
This is \eqref{eq:qm-psi-psi'-early-middle}.
Finally, the third line in (\ref{eq:q-s-bounds-middle}) and (\ref{20may2108}) imply:
\begin{equation}
  \label{eq:qm-E-integral}
  \int_\R \abs{E(t,x)} \psi(x)\qm(t,x) \d x \leq \frac{C}{a}
(t+1)^{{3\sqrt{N}}/{(2\lambda_*)}} e^{-\sqrt{N}a } \int_{\earlyt}^t (t
- s + 1)^{-3/ 2} e^{-{c(t_*-s)^2}/{a}} \d s.
\end{equation}
Since $(\abs{y} + 1)^{-3/ 2}$ is integrable, the integral on the right side
of \eqref{eq:qm-E-integral} is bounded by $C e^{-( t_*-t)^2/a}$.
Hence
\begin{equation*}
  \int_\R \abs{E(t,x)} \psi(x)\qm(t,x) \d x \leq \frac{C}{a}
(t+1)^{{3\sqrt{N}}/{(2\lambda_*)}} e^{-\sqrt{N}a } e^{-{c( t_*-t)^2}/{a}}.
\end{equation*}
We have thus verified \eqref{eq:qm-E-early-middle} and completed the proof of Lemma~\ref{lem:qm-early-middle}.
\end{proof}

\subsection*{The proof of Lemma~\ref{lem-may1506}}

We wish to understand $I(t)$ and $\dot{I}(t)$ given by \eqref{20mar2626} and \eqref{mar1912}, respectively.
Recall that $r=p+\qe+\qm$.
We use Lemmas~\ref{lem-may1902}, \ref{lem:early-corrector}, and \ref{lem:qm-early-middle} to control the terms with $p$, $\qe$, and $\qm$, respectively.
These lemmas yield
\begin{equation*}
  I(t)=\int_\R \psi(x) r(t,x) \d x = \Ufactor \gamma_*e^{-\gamma_*\bar x_0}(t + 1)^{{3\sqrt{N}}/{(2\lambda_*)}} 
  e^{-\sqrt{N}a}\Big[1 +  O\Big(e^{-\frac{c(t_*-t)^2}{a}}\Big)\Big]
\end{equation*}
and
\begin{equation*}
  \bal
  \frac{\d I}{\d t}(t)&=
  \farc{3}{2\lambda_*(t+1)} \int_\R {\psi'(x) r(t,x)} \d x +
  \int_\R E(t,x) \psi(x) r(t,x) \d x\\
  &= \farc{3}{2\lambda_*(t+1)}\sqrt{N}\Ufactor \gamma_*e^{-\gamma_*\bar x_0}(t + 1)^{{3\sqrt{N}}/{(2\lambda_*)}} 
  e^{-\sqrt{N}a} \Big[1 + O\Big(e^{-\frac{c(t_*-t)^2}{a}}\Big)\Big].
  \enbal
\end{equation*}
Finally, the lower bound \eqref{20may1925} and \eqref{eq:q-nonneg} imply that ${\dot I(t)}/{I(t)}$ doesn't become singular as $t \to t_*$.
Therefore
\begin{equation*}
  \frac{\dot I(t)}{I(t)} = \frac{3\sqrt{N}}{2\lambda_*(t + 1)} 
  \left[1 + O\left(e^{-\frac{c(t_*-t)^2}{a}}\right)\right]
\end{equation*}
when $t \in [0, t_*]$, and the proof of Lemma~\ref{lem-may1506} is complete.\hfill\qed

\section{Middle times}
\label{sec:late-modern}

Here, we consider the time interval $t \in [t_*, a]$.
Now, the story changes: the main corrector $\qm(t,x)$ becomes the dominant term 
in $I(t)$ and $\dot{I}(t)$, though, of course, the homogeneous term $p(t,x)$ 
is comparable to it when $t - t_* = O\big(\sqrt{a}\big)$.
Again, Lemma~\ref{lem:early-corrector} shows that the contributions of $\qe$ are negligible relative to those of $\qm$ in Lemma~\ref{lem:qm-late-middle} below.

We will prove the following.
\begin{lemma}
  \label{lem-may2202}
  There exist $c>0$ and $C>0$ such that for all  $t \in [t_*, a]$,
  \begin{equation}\label{eq:ratio-late-middle}
    \frac{|\dot{I}(t)|}{I(t)} \leq \frac{C}{a} 
    \left[(t - t_* + 1)^{-1/2} +e^{-{c(t - t_*)^2}/{a}}\right].
  \end{equation}
\end{lemma}
After integration, Lemma~\ref{lem-may2202} implies that
\begin{equation}\label{eq-middle-int}
\log\frac{I(a)}{I(t_*)}= O(a^{-1/2}).
\end{equation}

\subsection*{The free contribution}

We first bound the contributions of the free term $p$.
\begin{lemma}
  \label{lem:p-late-middle}
  There exist $c > 0$ and $C > 0$ such that for all $t \in [t_*, a]$,
  \begin{equation*}
    \int_\R [\psi(x) + \psi'(x)] p(t, x) \d x \leq C (t + 1)^{3\sqrt{N}/(2\lambda_*)} e^{-\sqrt{N} a} e^{-\frac{c(t - t_*)^2}{a}}
  \end{equation*}
  and
  \begin{equation*}
    \int_\R \abs{E(t, x)} p(t, x) \d x \leq \frac{C}{a} (t + 1)^{3\sqrt{N}/(2\lambda_*)} e^{-\sqrt{N} a} e^{-\frac{c(t - t_*)^2}{a}}.
  \end{equation*}
\end{lemma}
\begin{proof}
Recall expression \eqref{eq:p-zero} for $p(t,x)$.
We consider $(\psi + \psi')p$ separately on $\R_\pm$.
When $x\le 0$, we use $\psi(x)+\psi'(x)\le C\exp(\sqrt N x)$ from
(\ref{eq:adjoint-boundary}), and obtain
\begin{equation}
  \label{20may2206}
  [\psi(x)+\psi'(x)] p(t, x) \leq C \Lambda(t;a) 
  \exp\left[\left(\sqrt{N} - \frac{a}{2t}\right)x\right]g(t,x) \quad \textrm{for } x \leq 0.
\end{equation}
Now $t\ge t_*$, so $\sqrt N-a/(2t)\ge0$.
Also, \eqref{eq:g} shows that $\int_\R g(t,x)\d x= \sqrt{4 \pi t}$.
It follows that
\begin{equation}\label{20may2220}
\int_{\R_-} [\psi(x)+{\psi'}(x)] p(t, x) \d x 
\le C \Lambda(t;a) \int_\R g(t, x) \d x \leq C \sqrt{t} \Lambda(t; a) \quad \textrm{for all } t \geq t_*.
\end{equation}
When $x \geq 0$, we use $\psi(x)+\psi'(x)\le C(1+x)$ from
(\ref{eq:adjoint-boundary}) and $t \leq a$ to obtain
\begin{equation}\label{20may2212}
[\psi(x)+\psi'(x)] p(t, x) \leq C (1+x)\Lambda(t;a) e^{-ax/(2t)}
\leq C(1+x) \Lambda(t;a) e^{-x/2} \quad \textrm{for } x \geq 0.
\end{equation}
So
\begin{equation}\label{lmtpsipos}
\int_{\R_+} [\psi(x)+\psi'(x)] p(t, x) \d x \leq C \Lambda(t;a).
\end{equation}

We now turn to the error term with $E$.
Recalling the bound \eqref{20may1932} on $E$, \eqref{20may2206} and \eqref{20may2212} imply
\begin{equation}
  \label{eq:pwavelambda}
  \int_\R |E(t,x)|\psi(x)p(t,x)\d x \leq \frac{C \Lambda(t; a)}{\sqrt{t + 1}}  \int_\R e^{-c \abs{x}} \d x \leq \frac{C \Lambda(t; a)}{\sqrt{t}}.
\end{equation}
The lemma follows from \eqref{20may2220}, \eqref{lmtpsipos} and 
\eqref{eq:pwavelambda}
using \eqref{20may2102} with $a/t\le a/t_*=2\sqrt N$.
\end{proof}

\subsection*{The main corrector contribution}

We now estimate the contribution of $\qm(t,x)$ on the interval $[t_*,a]$.
\begin{lemma}
  \label{lem:qm-late-middle}
There exist $c>0$ and $C>0$ such that for all $t\in[t_*,a]$,
\begin{align}\label{20may2228}
\int \psi'(x)\qm(t,x) \d x 
&\le  {C} (t + 1)^{{3\sqrt{N}}/{(2\lambda_*)}}e^{-\sqrt{N}a}
\Big[(t-t_* + 1)^{-1/2}+ e^{-{c(t - t_*)^2}/{a}}\Big],
\\
\label{20may2232}
\int \abs{E(t,x)}\psi(x)\qm(t,x)\d x 
&\le \farc{C}{a} (t + 1)^{{3\sqrt{N}}/{(2\lambda_*)}}e^{-\sqrt{N}a}
\Big[(t-t_* + 1)^{-1/2}+ e^{-{c(t - t_*)^2}/{a}}\Big].
\end{align}
Moreover, for all $t\ge t_*$,
\begin{equation}
I(t) \geq \int_\R \psi(x)\qm(t, x) \d x
\ge ca^{{3\sqrt{N}}/{(2\lambda_*)}}e^{-\sqrt{N}a}.
\label{20may2244}
\end{equation}
\end{lemma}
We emphasize that the last bound holds for all $t \geq t_*$, and thus extends to the late times after $a$ as well.

\begin{proof}
  We again represent $\qm(t,x)$ via the Duhamel formula (\ref{20may2216}) as an integral of $q^s$ satisfying \eqref{eq:q-s}.
  By Lemma~\ref{lem:qs-integrals},
\begin{equation}
  \label{20may2218}
  \begin{aligned}
    &\int_\R \psi'(x)\qm(t,x) \d x \leq C \int_{\earlyt}^t \Lambda(s;a) (t - s + 1)^{-1/ 2} \d s,\\
    &\int_\R \abs{E(t,x)}\psi(x)\qm(t,x) \d x \leq C \int_{\earlyt}^t 
    \Lambda(s;a) a^{- 1 /2} (t - s + 1)^{-3/ 2} \d s.
  \end{aligned}
\end{equation}
By (\ref{20may2108}),
\begin{equation}
  \label{20may2222}
  \bal
  \Lambda(s;a) \le\frac{C}{\sqrt{a}} (t + 1)^{{3\sqrt{N}}/{(2\lambda_*)}}
  e^{-\sqrt{N}a}e^{-c(s-t_*)^2/a}
  \enbal
\end{equation}
for $\earlyt \le s\le t\le a$.
In light of \eqref{20may2218}, we must bound integrals of the form
\begin{equation}\label{eq:conv}
Z_\alpha(t):=\int_{\earlyt}^t (t - s + 1)^{-\alpha} e^{-{c(s - t_*)^2}/{a}} \d s
\le
\int_{0}^\infty (s + 1)^{-\alpha} e^{-{c(t - t_*-s)^2}/  {a}} \d s
\end{equation}
for $\alpha \in \{1/2, 3/2\}$.

We cut the integral in the right side of \eqref{eq:conv} at
$s=(t-t_*)/2$:
\begin{equation*}
\int_{\frac{t-t_*}2}^\infty  (s + 1)^{-\alpha} e^{-{c(t - t_*-s)^2}/  {a}}
\d s
\le
\Big(\frac{t-t_*}2+1\Big)^{-\alpha}\int_{\frac{t-t_*}2}^\infty e^{-{c(t - t_*-s)^2}/  {a}} 
\d s
\le C\sqrt a (t-t_*+1)^{-\alpha}
\end{equation*}
and
\begin{equation*}
\begin{aligned}
  \int_0^{\frac{t-t_*}2} (s + 1)^{-\alpha} e^{-{c(t - t_*-s)^2}/  {a}}
  \d s
  \le
  \int_0^{\frac{t-t_*}2} e^{-{c(t - t_*-s)^2}/  {a}} \d s
  &\le
  \int_{\frac{t - t_*}2}^{\infty} e^{-{cr^2}/  {a}} \d r\\
  &\le C\sqrt a e^{-c(t-t_*)^2/(4a)}.
\end{aligned}
\end{equation*}
we conclude that
\be\label{20may2226}
Z_\alpha(t)\le C\sqrt a\Big[(t-t_* + 1)^{-\alpha}+ e^{-{c(t - t_*)^2}/ {a}}\Big].
\ee
When $\alpha>1$, the integrability of $(s+1)^{-\alpha}$ gives an
alternative bound. Still cutting  at $s=(t-t_*)/2$, we now write
\begin{equation*}
\int_{\frac{t-t_*}2}^\infty  (s + 1)^{-\alpha} e^{-{c(t - t_*-s)^2}/  {a}}
\d s
\le
\int_{\frac{t-t_*}2}^\infty (s + 1)^{-\alpha} 
\d s
\le C (t-t_*+1)^{1-\alpha}
\end{equation*}
and
\begin{equation*}
\int_0^{\frac{t-t_*}2} (s + 1)^{-\alpha} e^{-{c(t - t_*-s)^2}/  {a}}
\d s
\le e^{-c(t-t_*)^2/(4a)}
\int_0^{\infty} (s + 1)^{-\alpha}  \d s
\le C e^{-c(t-t_*)^2/(4a)}.
\end{equation*}
We conclude that
\be Z_\alpha(t)\le C\Big[(t-t_* + 1)^{1-\alpha}+ e^{-{c(t - t_*)^2}/
{a}}\Big]\qquad\text{if $\alpha>1$}.
\label{20may2230}
\ee

Combining (\ref{20may2218}), (\ref{20may2222}), and (\ref{20may2226}) with
$\alpha=1/2$ gives
\begin{align*}
\begin{split}
\int_\R \psi'(x)\qm(t,x) \d x &\le \frac{C}{\sqrt{a}} (t + 1)^{{3\sqrt{N}}/{(2\lambda_*)}} e^{-\sqrt{N}a}Z_{1/2}(t) \\
&\le  {C} (t + 1)^{{3\sqrt{N}}/{(2\lambda_*)}}e^{-\sqrt{N}a} \Big[(t-t_* + 1)^{-1/2}+ e^{-{c(t - t_*)^2}/{a}}\Big],
 \end{split}
\end{align*}
while (\ref{20may2218}), (\ref{20may2222}), and (\ref{20may2230}) with
$\alpha=3/2$ imply
\begin{align*}
\begin{split}
\int_\R \abs{E(t,x)}\psi(x)\qm(t,x)\d x &\le \frac{C}{a} (t + 1)^{{3\sqrt{N}}/{(2\lambda_*)}} e^{-\sqrt{N}a}Z_{3/2}(t)\\
&\le \farc{C}{a} (t + 1)^{{3\sqrt{N}}/{(2\lambda_*)}}e^{-\sqrt{N}a} \Big[(t-t_* + 1)^{-1/2}+ e^{-{c(t - t_*)^2}/{a}}\Big].
\end{split}
\end{align*}
We have thus confirmed \eqref{20may2228} and \eqref{20may2232}.

Finally, we need a lower bound on $I(t)$.
With $p\ge0$ and $\qe\ge0$, we can write for $t\ge t_*$
\be\label{1stboundI}
I(t)\ge \int_\R \qm(t,x)\psi(x)\d x
=
\int_{\earlyt}^{t} \int_\R q^s(t,x) \psi(x)\d x\d s
\ge
\int_{\earlyt}^{t_*} \int_{\R_+} q^s(t,x) \psi(x)\d x\d s,
\ee
where we used again the Duhamel formula (\ref{20may2216}) and $q^s\ge0$.
We recall that $q^s$ satisfies \eqref{eq:q-s}.
We need a lower bound on $q^s$, so we look for a subsolution to
\eqref{eq:q-s} on $\R_+$.

We first focus on the initial condition $q^s(s,x)=[N-V(s,x)]p(s,x)$ for $x\in\R_+$.
Recall the definition \eqref{20may1214} of $V$.
The comparison principle implies that $H$ is decreasing in $x$, so
\begin{equation}
  \label{eq:H-monotone}
  H(s, x + m(s)) \leq H(s, m(s)) \quad \textrm{for all } x \geq 0.
\end{equation}
By \eqref{20mar1908} and $s \geq \earlyt = \earlyxi a$, $H(s, x + m(s))$ is very close to $U_0(0) < 1$ provided $a$ is sufficiently large.
We can thus assume that $H(s, m(s)) \leq 1 - c$ for all $s \geq \earlyt$.
Using \eqref{eq:H-monotone} and \eqref{eq:nonlin-deriv}, it follows that $N - V(s, x) \geq c$ in this region.
That is,
\begin{equation}
  \label{eq:q-s-lower}
  q^s(s,x) \geq c p(s,x) \quad \textrm{for all } s \geq \earlyt, \, x \geq 0.
\end{equation}
Going back to \eqref{eq:p-zero}, we see that for all $x\ge 0$ and $s\ge
\earlyt$, we have
\begin{equation}
  \label{eq:p-lower-prelim}
  p(s, x) \geq \Lambda(s;a) e^{-Cx} \exp\Big\{-\frac{1}{4s}\Big[x - \frac{3}{2\lambda_*}\log(s+1)\Big]^2\Big\}.
\end{equation}
We are free to assume $s \geq \earlyt \geq 2$, so a variation on \eqref{eq:gaussian mean bound} yields
\begin{equation*}
  \frac{1}{4s}\Big[x - \frac{3}{2\lambda_*}\log(s+1)\Big]^2 \leq \frac{x^2}{8 (1 - \eps)} + C(\eps) \quad \textrm{for all } s \geq \earlyt, \, x \in \R
\end{equation*}
and $\eps \in (0, 1)$.
Hence \eqref{eq:p-lower-prelim} and $e^{-Cx} \geq C(\eps) e^{-\eps x^2}$ implies
\begin{equation}
  \label{eq:p-lower}
  p(s, x) \geq c \Lambda(s;a) e^{-Cx} e^{-x^2/7} \geq c \Lambda(s;a) e^{-x^2/6} \geq c \Lambda(s;a) x e^{-x^2/4},
\end{equation}
where we have allowed $c > 0$ to change from expression to expression.
We now define
\begin{equation*}
  \varphi(\lambda, x) = \frac{x}{\lambda^{3/2}} \exp\left(-\frac{x^2}{4\lambda}\right)
\end{equation*}
and note for later reference that
\begin{equation}
  \label{eq:normalization}
  \int_{\R_+} x \varphi(\lambda, x) \d x = 2\sqrt{\pi}~~\hbox{ for all $\lambda > 0$.}
\end{equation}
Combining \eqref{eq:q-s-lower} and \eqref{eq:p-lower}, we can write
\begin{equation}
  \label{eq:q-s-lower-varphi}
  q^s(s,x) \geq c \Lambda(s;a) \varphi(1, x) \quad \textrm{for all } s \geq \earlyt, \, x \geq 0.
\end{equation}

Now, we consider the PDE in \eqref{eq:q-s}.
Since we are looking for a lower bound on $q^s(t,x)$, 
we cannot neglect the negative term~$-V(t,x)q^s$ in 
the right side of (\ref{eq:q-s}).
By Lemma~\ref{lem:front-decay},
\be
\label{20may2234}
V(t,x)\le Be^{-cx}
\ee
for all $t>0$ and $x>0$.
To obtain a subsolution,
we are free to impose a Dirichlet condition at~$x = 0$.
We let $\ubar v^s(t,x)$ solve
\begin{equation*}
  \bal
  &   \pdr{\ubar v^s }{t}= \pdrr{\ubar v^s}{x} - \frac{3}{2\lambda_*(t + 1)} \pdr{\ubar v^s}{x} - 
  B e^{-cx} \ubar v^s ~~~\textrm{ for $t>s$ and $x>0$,} \\
  &    \ubar v^s(t, 0) = 0~~\hbox{ for $t>s$},\\
  &   \ubar v^s(s, x) = \varphi(1,x)~~\hbox{ for $x>0$}.
  \enbal
\end{equation*}
By \eqref{eq:q-s-lower-varphi} and \eqref{20may2234}, we have
\begin{equation*}
q^s(t,x)\ge c\Lambda(s;a)\ubar v^s(t,x) \quad \textrm{for all } t \geq s \geq \earlyt, \, x \geq 0.
\end{equation*}
Then, with \eqref{1stboundI},
\be\label{eq:q-lower}
I(t)\ge c \int_{\earlyt}^{t_*}\Lambda(s;a)\int_{\R_+}\ubar v^s(t,x)\psi(x)
\d x\d s.
\ee
The following lemma gives a lower bound on $\ubar v^s(t,x)$. 
\begin{lemma}\label{lem:Dirichlet-lower}
There exists $c > 0$ such that for all $t\ge s$ and $x>0$ we have
\begin{equation}
  \label{20may2242}
\ubar v^s(t, x) \geq c \, \varphi(t - s + 1,x) + 
 \ubar{R}(t, s, x)
\end{equation}
with
\begin{equation}
  \label{eq:R-lower}
  \abs{\ubar{R}(t, s, x)} \leq C (t - s + 1)^{-2} x e^{-x^2/[8(t - s + 1)]}
	=C (t-s+1)^{-1/2}\varphi\big(2(t - s + 1), x\big).
\end{equation}
\end{lemma}
\begin{proof}
This is nearly Lemma~2.2 in \cite{HNRR}.
The only difference is the exponentially decaying potential, which is 
negligible on the scale~${x \sim \sqrt{t}}$, where the analysis really happens.
\end{proof}

Using (\ref{eq:adjoint-boundary}), we have $\psi(x)\ge c x$ for $x\ge0$.
Then (\ref{20may2242}), \eqref{eq:R-lower}, and \eqref{eq:normalization} imply
\begin{equation*}
  \begin{aligned}
    \int_{\R_+}\ubar v^s(t,x)\psi(x) \d x &\ge c \int_{\R_+} x \varphi(t-s+1,x) \d x - C (t - s + 1)^{-1/2} \int_\R x \varphi\big(2(t - s + 1), x\big) \d x\\
    &\ge c - C(t - s + 1)^{-1/2},
  \end{aligned}
\end{equation*}
where we have allowed $c$ to change from expression to expression.
So $\int_{\R_+}\ubar v^s \psi$ is uniformly positive once $t - s \geq C'$ for $C'$ large.
On the other hand, $\ubar{v}^s$ is positive, so the integral is positive on the time interval $[s, s + C']$.
By compactness, it follows that
\begin{equation*}
  \int_{\R_+}\ubar v^s(t,x)\psi(x) \d x \geq c \quad \textrm{for all } t \geq s \geq \earlyt.
\end{equation*}
Now \eqref{eq:q-lower} yields
\begin{equation}
  \label{eq:I-lower-integral}
  I(t) \geq c \int_{\earlyt}^{t_*} \Lambda(s;a) \d s
\end{equation}
for all $t\ge t_*$.
For $s \in [\earlyt, t_*]$, \eqref{eq:Lambda} implies
\begin{equation}
  \label{eq:Lambda-lower-prelim}
  \Lambda(s; a) \geq \frac{c}{\sqrt{a}} a^{3\sqrt{N}/(2\lambda_*)} e^{-a\theta(s/a)}
\end{equation}
for $\theta(\xi) = N\xi + 1/(4\xi)$.
We recall that $\theta$ is strictly convex and attains its minimum of $\sqrt{N}$ at
{$\xi_* = 1/(2\sqrt{N}) = t_*/a$}.
Also, $\earlyt = \earlyxi a.$
Since $\theta$ is smooth on the interval $[\earlyxi, \xi_*]$, there exists $C > 0$ such that
\begin{equation*}
  \theta(\xi) \leq \sqrt{N} + C(\xi - \xi_*)^2 \quad \textrm{for all } \xi \in [\earlyxi, \xi_*].
\end{equation*}
Then \eqref{eq:Lambda-lower-prelim} yields
\begin{equation*}
  \Lambda(s; a) \geq \frac{c}{\sqrt{a}} a^{3\sqrt{N}/(2\lambda_*)} e^{-\sqrt{N} a} e^{-C(t_* - s)^2/a}
\end{equation*}
and \eqref{eq:I-lower-integral} implies
\begin{equation*}
  I(t) \geq c a^{3\sqrt{N}/(2\lambda_*)} e^{-\sqrt{N} a} \int_{\earlyt}^{t_*} e^{-C(t_* - s)^2/a} \frac{\dn s}{\sqrt{a}} \geq c a^{3\sqrt{N}/(2\lambda_*)} e^{-\sqrt{N} a}
\end{equation*}
for all $t \geq t_*$.
This completes the proof of Lemma~\ref{lem:qm-late-middle}.
\end{proof}
 
\subsection*{The proof of Lemma~\ref{lem-may2202}}
 
Gathering together Lemmas~\ref{lem:early-corrector}, \ref{lem:p-late-middle}, and \ref{lem:qm-late-middle}, we obtain
\begin{equation*}
\bal
|\dot{I}(t)| &\leq \frac{3}{2(t + 1)}\int_\R\psi'(x)r(t,x) \d x + \int_\R \abs{E(t,x)}\psi(x)r(t,x) \d x\\
&\leq \farc{C}{a} (t + 1)^{{3\sqrt{N}}/{(2\lambda_*)}}e^{-\sqrt{N}a}
\Big[(t-t_* + 1)^{-1/2}+ e^{-{c(t - t_*)^2}/{a}}\Big].
\enbal
\end{equation*}
Taking into account \eqref{20may2244}, we see that (\ref{eq:ratio-late-middle}) follows.\hfill\qed

\section{The late times}
\label{sec:postmodern}

We finish with the times $t \geq a$.
In this regime, $p$ and $\qe$ should be exponentially negligible.
However, since this time period is unbounded, we must take care to ensure that $ {\dot{I}(t)}/{I(t)}$ is integrable in time, and, in fact, small.
We will prove the following lemma.
\begin{lemma}\label{lem-may2402}
There exists $C>0$ such that for all $t\ge a$ we have
\be\label{20may2420}
\farc{|\dot I(t)|}{I(t)}\le \farc{C}{t^{3/2}}.
\ee
\end{lemma}
After integration, Lemma~\ref{lem-may2402} implies that
\begin{equation}\label{eq-late-int}
\log\frac{I(\infty)}{I(a)}= O(a^{-1/2}).
\end{equation}
Note that the lower bound (\ref{20may2244}) still holds, so it suffices to show that
\begin{equation*}
|\dot I(t)| \leq \farc{C}{t^{3/2}} a^{3\sqrt{N}/(2\lambda_*)} e^{-\sqrt{N} a} \quad \textrm{for }t \geq  a.
\end{equation*}

\subsection*{The free contribution}

It is simple to control the free contribution at late times, since it has decayed into irrelevance.
\begin{lemma}
  \label{lem:p-late}
  There exist $c > 0$ and $C > 0$ such that for all $t \geq a$,
  \begin{equation}
    \label{eq:p-late}
    \int_\R \big[\psi(x) + \psi'(x) + \abs{E(t, x)}\psi(x)\big] p(t, x) \d x \leq C e^{-(\sqrt{N} + c)a} e^{-ct}.
  \end{equation}
\end{lemma}
\begin{proof}
For $x\le 0$ we recall that $E$ is bounded and we simply use (\ref{20may2220}) and \eqref{eq:Lambda}, which give
\begin{equation}
  \label{20may2424}
\int_{\R_-} \big[\psi(x)+{\psi'}(x)+\abs{E(t,x)}\psi(x)\big] p(t, x)\d x \le C \sqrt{t} \Lambda(t;a) \leq Ct^C e^{-a\theta(t/a)}
\end{equation}
for $\theta$ defined in \eqref{eq:rate-function}.
Now, $\min\theta = \sqrt{N}$ is uniquely attained at $\xi_* = 1/(2 \sqrt{N}) < 1$, and $\theta(\xi) \geq N \xi$.
If $c > 0$ is sufficiently small, it follows that
\begin{equation}
  \label{eq:theta-late}
  \theta(\xi) \geq \sqrt{N} + c + c \xi \quad \textrm{for all } \xi \geq 1.
\end{equation}
Together with \eqref{20may2424}, this implies
\begin{equation}
  \label{eq:p-late-left}
  \int_{\R_-} \big[\psi(x)+{\psi'}(x)+\abs{E(t,x)}\psi(x)\big] p(t, x)d x \le C e^{-(\sqrt{N} + c)a} e^{-ct}.
\end{equation}
For $x \ge 0$, recall that $\psi(x)\le C(1+x)$ and $\psi'(x)\le C$. Then \eqref{eq:p-zero} implies
\begin{equation*}
  \int_{\R_+} \big[\psi(x)+\psi'(x)+\abs{E(t,x)}\psi(x)\big] p(t, x) \d
x \leq C \Lambda(t;a)\int_{0}^\infty (1+x) e^{-{ax}/{(2t)}} \d x \leq {C} t^C e^{-a \theta(t/a)}.
\end{equation*}
Again, \eqref{eq:theta-late} yields
\begin{equation*}
  \int_{\R_+} \big[\psi(x)+{\psi'}(x)+\abs{E(t,x)}\psi(x)\big] p(t, x)d x \le C e^{-(\sqrt{N} + c)a} e^{-ct}.
\end{equation*}
Combining this with \eqref{eq:p-late-left}, we obtain \eqref{eq:p-late}.
\end{proof}
 
\subsection*{The main corrector contribution}

Next, we control $\qm$ at late times.
\begin{lemma}
  \label{lem:qm-late}
  There exist $c > 0$ and $C > 0$ such that for all $t \geq a$,
  \begin{equation}
    \label{eq:qm-psi'-late}
    \int_\R \psi'(x) \qm(t, x) \d x \leq \frac{C}{\sqrt{t}} a^{3\sqrt{N}/(2\lambda_*)} e^{-\sqrt{N} a}
  \end{equation}
  and
  \begin{equation}
    \label{eq:qm-E-late}
    \int_\R \abs{E(t, x)} \psi(x) \qm(t, x) \d x \leq \frac{C}{t^2} a^{3\sqrt{N}/(2\lambda_*)} e^{-\sqrt{N} a}.
  \end{equation}
\end{lemma}

\begin{proof}
  We recall the Duhamel formula (\ref{20may2216}) for $q^s$, the solution to \eqref{eq:q-s}. We also recall Lemma~\ref{lem:qs-integrals}: for $t\ge s\ge\earlyt$,
\be\label{eq:q-s-bounds-late}
\bal
&\int_\R {\psi'}(x)q^s(t,x) \d x\leq C \frac{s^2}{a^2}\Lambda(s;a)
(t - s + 1)^{- 1/ 2},\\
&\int_\R \abs{E(t,x)}{\psi}(x)q^s(t,x) \d x
\leq C \frac{s^2}{a^2} \Lambda(s;a) t^{- 1/ 2} (t - s + 1)^{-3/ 2}.
\enbal
\ee
To bound the contributions of $\qm$, we integrate over $s \in [\earlyt, t]$.
We rely on two different estimates for $\Lambda$.
When $s \in [\earlyt, a]$, \eqref{20may2108} implies:
\begin{equation}
  \label{eq:Lambda-s-middle}
  \frac{s^2}{a^2} \Lambda(s; a) \leq \frac{C}{\sqrt{a}} a^{3\sqrt{N}/(2 \lambda_*)} e^{-\sqrt{N} a} e^{-\frac{c(s - t_*)^2}{a}} \quad \textrm{for } s \in [\earlyt, a].
\end{equation}
This estimate cannot hold for $s \geq a$, since $\theta$ is not \emph{uniformly} convex.
Nonetheless, \eqref{eq:theta-late} yields:
\begin{equation}
  \label{eq:Lambda-s-late}
  \frac{s^2}{a^2}\Lambda(s; a) \leq C s^2 a^C e^{-a \theta(a/s)} \leq C e^{-(\sqrt{N} + c)a} e^{-cs} \quad \textrm{for } s \in [a, \infty).
\end{equation}

We first use \eqref{eq:q-s-bounds-late} and \eqref{eq:Lambda-s-middle} to control the contributions of $q^s$ when $s \in [\earlyt, a]$.
Recalling the definition \eqref{eq:conv} of $Z_\alpha(t)$ and noticing
from \eqref{20may2226} that $Z_{1/2}(t)\le C\sqrt{a/t}$ for $t\ge a$, we find
\be
\label{eq:q-s-psi'-middle-late}
\bal
\int_{\earlyt}^{a} \int_\R \psi'(x)q^s(t,x) \d x \ds s &\leq
\frac{C}{\sqrt{a}}
a^{{3\sqrt{N}}/{(2\lambda_*)}}
e^{-\sqrt{N}a}
\int_{\earlyt}^a (t - s + 1)^{- 1/ 2}e^{-c(s-t_*)^2/a} \d s\\
&
\le \frac{C}{\sqrt{t}}
a^{{3\sqrt{N}}/{(2\lambda_*)}}
e^{-\sqrt{N}a}.
\enbal
\ee
Similarly, noticing from \eqref{20may2226} that
$Z_{3/2}(t)\le C\sqrt a/t^{3/2}$, we obtain
\be\label{eq:q-s-E-middle-late}
\bal
\int_{\earlyt}^{a} \int_\R \abs{E(t, x)} \psi(x)q^s(t,x) \d x \ds s
&\leq \farc{C}{\sqrt{at}}a^{{3\sqrt{N}}/{(2\lambda_*)}}
e^{-\sqrt{N}a}\int_{\earlyt}^a (t - s + 1)^{-3/ 2}e^{-c(s-t_*)^2/a} \d s\\
&\le \farc{C}{t^2}a^{{3\sqrt{N}}/{(2\lambda_*)}}
e^{-\sqrt{N}a}.
\enbal
\ee
Next, we use \eqref{eq:q-s-bounds-late} and \eqref{eq:Lambda-s-late} to control the contributions of $q^s$ when $s \geq a$.
To do so, we rely on the following bound: for each $\al \geq 0$ and $t > 0$,
\begin{equation}
  \label{eq:poly-exp-integral}
  \int_0^t e^{-cs} (t - s + 1)^{-\al} \d s \leq \int_0^{t/2} e^{-cs} \left(\frac{t}{2} + 1\right)^{-\al} \d s + \int_{t/2}^t e^{-cs} \d s \leq C_\al (t + 1)^{-\al}.
\end{equation}
Thus
\begin{equation}
  \label{eq:q-s-psi'-late-late}
    \int_a^t \int_\R \psi'(x)q^s(t,x) \d x \ds s \leq
    C e^{-(\sqrt{N} + c) a} \int_{a}^t (t - s + 1)^{- 1/ 2} e^{-cs} \d s \leq \frac{C}{\sqrt{t}} e^{-(\sqrt{N} + c) a}
\end{equation}
and
\begin{equation}
  \label{eq:q-s-E-late-late}
    \int_a^t \int_\R \abs{E(t, x)} \psi(x) q^s(t,x) \d x \ds s \leq
    \frac{C}{\sqrt{t}} e^{-(\sqrt{N} + c) a} \int_{a}^t (t - s + 1)^{- 3/ 2} e^{-cs} \d s \leq \frac{C}{t^2} e^{-(\sqrt{N} + c) a}.
\end{equation}
We recall that
\begin{align*}
  &\int_\R \psi'(x) \qm(t, x) \d x = \int_{\earlyt}^t \int_\R \psi'(x)q^s(t,x) \d x \ds s,\\
  &\int_\R \abs{E(t, x)} \psi(x) \qm(t, x) \d x = \int_{\earlyt}^t \int_\R \abs{E(t, x)} \psi(x) q^s(t,x) \d x \ds s.
\end{align*}
Therefore \eqref{eq:q-s-psi'-middle-late} and \eqref{eq:q-s-psi'-late-late} imply \eqref{eq:qm-psi'-late}, while \eqref{eq:q-s-E-middle-late} and \eqref{eq:q-s-E-late-late} imply \eqref{eq:qm-E-late}.
\end{proof}

\subsection*{The proof of Lemma~\ref{lem-may2402}}

First, we recall (\ref{20may2244}):
\begin{equation}\label{20may2438}
I(t)\ge
ca^{{3\sqrt{N}}/{(2\lambda_*)}}e^{-\sqrt{N}a}~~\hbox{for $t \geq t_*$.}
\end{equation}
Next, collecting Lemmas~\ref{lem:early-corrector}, \ref{lem:p-late}, and \ref{lem:qm-late},
we see that (\ref{mar1912}) gives
\begin{equation}\label{20may2439}
|\dot I(t)| \leq
\farc{3}{2\lambda_*(t+1)} \int_\R  \psi'(x) r(t,x) \d x +
\int_\R \abs{E(t,x)} r(t,x)\psi(x) \d x \le \farc{C}{t^{3/2}}a^{{3\sqrt{N}}/{(2\lambda_*)}}e^{-\sqrt{N}a}.
\end{equation}
We obtain (\ref{20may2420}) from \eqref{20may2438} and \eqref{20may2439}.\hfill\qed

\subsection*{Proof of the main result}

As we have mentioned, Proposition~\ref{prop:main} is an immediate consequence of 
Lemmas~\ref{lem-may1506}, \ref{lem-may2202}, and~\ref{20may2420}.
Combining
\eqref{eq-early-int},
\eqref{eq-middle-int} and
\eqref{eq-late-int}, we find
\begin{equation*}
\log\frac{I(\infty)}{I(0)}=
\frac{3\sqrt{N}}{2\lambda_*}\log\frac{a}{2\sqrt N}+ O\big(a^{-1/2}\big).
\end{equation*}
Exponentiating, we obtain Proposition~\ref{prop:main}.\hfill\qed

\section{The early corrector --- Proof of Lemma~\ref{lem:early-corrector}}
\label{sec:qe}

Given $\eps_0 \in (0, 1)$, to be chosen later, Lemma~\ref{lem:front-decay} ensures the existence of~$K \geq 0$ such that
\begin{equation*}
N(1 - \eps_0)  \leq V(t,x)\hbox{ for all $t>0$ and $x<-K$.}
\end{equation*}
The same lemma implies that
\begin{equation*}
  N-V(t,x) \leq Be^{\gamma_*x}.
\end{equation*}
Hence if $\ti q$ solves
\begin{equation*}
\pdr{\ti q}{t} = \pdrr{\ti q}{x} - \frac{3}{2\lambda_*(t+1)} \pdr{\ti q}{x} - 
N (1 - \eps_0) \one_{(-\infty, -K)}(x) \ti q+ B e^{\gamma_* x}\one_{[0,\earlyt]}(t)p, \quad \ti q(0, x) = 0,
\end{equation*}
comparison with \eqref{eq:scattering-early} implies 
\begin{equation*}
\qe(t,x) \leq \ti q(t,x).
\end{equation*}

To estimate $\ti q$, we decompose it into two parts, $\rho$ and $\sigma$
\begin{equation*}
 \qe(t,x) \leq \ti q(t,x) = \rho(t,x)+\sigma(t,x),
\end{equation*}
where $\rho$ is the solution to
\begin{equation}\label{eq:rho-ev}
\pdr{\rho}{t} = \pdrr{\rho}{x} - \frac{3}{2\lambda_*(t + 1)} \pdr{\rho}{x} - 
N(1 - \eps_0) \rho + B e^{\gamma_* x} \one_{[0,\earlyt]}(t)p, \quad \rho(0, x) = 0,
\end{equation}
and $\sigma$ the solution to
\begin{equation}\label{eq:sigma-ev}
\pdr{\sigma}{t} = \pdrr{\sigma}{x} - \frac{3}{2\lambda_*(t+1)} \pdr{\sigma}{x} 
-N(1 - \eps_0) \one_{(-\infty, -K)}(x) \sigma + N(1 - \eps_0) \one_{[-K, \infty)}(x) \rho, \quad \sigma(0, x) = 0.
\end{equation}
We will be able to estimate $\rho$ more or less explicitly.
For $\sigma$, we emphasize that the forcing in the right side of (\ref{eq:sigma-ev}) is supported on $[-K, \infty)$.
That is, unlike $r$ and $\qe$, $\sigma$ is never forced deep in $\R_-$.
There is thus no need to estimate it with a further corrector---we can
control it with Lemma~\ref{lem:super}.

Lemma~\ref{lem:early-corrector} is a consequence of the two following lemmas:
\begin{lemma}
  \label{lemmaaboutrho}
  There exist $C, c>0$ independent of $a$ and $t$ such that for all $a \geq 1$ and $t \geq 0$:
\be\int_\R \psi(x) \rho(t,x)\d x 
+  \int_\R \psi'(x)\rho(t,x)\d x
+  \int_\R \psi (x)\big|E(t,x)\big| \rho(t,x)\d x\le C e^{-(\sqrt N+c) a-ct}
\label{eq:lemmaaboutrho}
\ee
\end{lemma}

\begin{lemma}
  \label{lemmaaboutsigma}
  There exist $C, c > 0$ independent of $a$ and $t$ such that for all $a \geq 1$ and $t \geq 0$:
  \begin{align*}
    \begin{split}
      &\int_\R \psi(x) \sigma(t,x) \d x \leq Ce^{-(\sqrt{N}+c)a},\\
      &\int_\R \psi'(x) \sigma(t,x) \d x \leq \frac{C}{\sqrt{t + 1}} e^{-(\sqrt{N}+c)a},\\
      &\int_\R  \psi(x) \abs{E(t,x)} \sigma(t,x) \d x \leq \farc{C}{(t + 1)^2} e^{-(\sqrt{N}+c)a}.
    \end{split}
  \end{align*}
\end{lemma}

\subsection*{Proof of Lemma~\ref{lemmaaboutrho}}

We use the Duhamel formula to write the solution to (\ref{eq:rho-ev}) as
\be\label{20may1322}
\rho(t,x)=\int_0^{t\wedge \earlyt} \rho^s(t,x) \d s.
\ee
Here, $\rho^s(t,x)$ is the solution to
\begin{equation}\label{eq:rho-s}
\bal
&\pdr{\rho^s}{t} = \pdrr{\rho^s}{x} - \frac{3}{2\lambda_*(t + 1)} \pdr{\rho^s}{x}
-N(1 - \eps_0) \rho^s  \hbox{ for $t>s$},&\qquad
&\rho^s(s, x) = Be^{\gamma_* x} p(s, x).
\enbal
\end{equation}
Recall the expression \eqref{eq:explicit} of $p(t,x)$ and the definition \eqref{eq:mu} of $\shift(t)$:
\begin{equation*}
\shift(t)=-a+\frac3{2\lambda_*}\log(t+1).
\end{equation*}
This allows us to write the initial condition $\rho^s(s, x)$ as a Gaussian:
\begin{equation*}
\bal
\rho^s(s, x) = Be^{\gamma_*x} p(s, x) &=Be^{\gamma_*x} \frac 1 {\sqrt{4\pi
s}} \exp\Big\{-Ns - \frac{1}{4s}\big[x-\shift(s)\big]^2\Big\} \\
&=\frac{B} {\sqrt{4\pi s}}e^{\gamma_*(\shift(s)+\gamma_*s)} \exp\Big\{-Ns
- \frac{1}{4s}\big[x-\shift(s)-2\gamma_*s\big]^2\Big\}.
\enbal
\end{equation*}
One can then verify that
\begin{equation}\label{20may1430}
\rho^s(t, x) = \frac{B}{\sqrt{4\pi t}} e^{-N \eps_0 s} 
e^{-N(1 - \eps_0)t}e^{\gamma_*(\shift(s) + \gamma_* s)} 
\exp\left[-\frac{(x - \shift(t) - 2\gamma_* s)^2}{4t}\right]
\end{equation}
satisfies both the initial condition and (\ref{eq:rho-s}).
The Duhamel formula (\ref{20may1322}) implies
\begin{equation}\label{eq:mean-mean}
\int_\R \psi(x)\rho(t, x) \d x=\int_0^{t\wedge \earlyt}\int_\R \psi(x)\rho^s(t,x) \d x \ds s
\end{equation}
for all $t\ge0$.

To bound the first term in \eqref{eq:lemmaaboutrho}, it suffices to
show that for all $t\ge0$ and $s\le t\wedge \earlyt$,
\be\label{may1968}
\int_\R \psi(x)\rho^s(t,x)\d x \le C(1+t)^C e^{-(\sqrt N+c) a -ct}
\ee
for some $c>0$, $C>0$.
Then, the integral over $s\le \earlyt\le a$ in
\eqref{eq:mean-mean} gives at most a factor $a$ which can be absorbed,
together with the $(1+t)^C$ factor, into
 in the exponential
decay by making $c$ smaller.

We now show \eqref{may1968}.
It follows from (\ref{eq:adjoint-boundary}) that $\psi(x)\le C e^{\sqrt
N x}$ for some $C$; then, in (\ref{20may1430}), we obtain 
\be\label{20may1410x}
\psi(x)\rho^s(t,x)\le
e^{\sqrt N x}\rho^s(t,x)  \le
\frac{C}{\sqrt{t}} e^{-N(1 - \eps_0)t} 
e^{\gamma_*(\shift(s) + \gamma_* s)}
e^{\sqrt{N}x} \exp\left[-\frac{(x - \shift(t) - 2\gamma_* s)^2}{4t}\right].
\ee
Now,
\begin{equation}\label{20may1414x}
e^{\sqrt{N} x} \exp\left[-\frac{(x - \shift(t) - 2\gamma_* s)^2}{4t}\right] 
=e^{\sqrt{N}(\shift(t) + 2\gamma_* s + \sqrt{N} t)} 
\exp\Big[-\frac{(x - \shift(t) - 2\gamma_* s - 2\sqrt{N} t)^2}{4t}\Big].
\end{equation}
Hence \eqref{20may1410x} yields
\be\label{20may1412x}
\bal
\int_\R \psi(x)\rho^s(t,x)\d x &\le 
C\exp\big\{-N(1 - \eps_0)t+\gamma_*(\shift(s) + \gamma_* s)
+\sqrt{N}\big(\shift(t) + 2\gamma_* s + \sqrt{N} t\big)\big\}\\
&=C\exp\big\{\sqrt{N}\shift(t)+N\eps_0t+\gamma_*\shift(s)+\gamma_*\big(\gamma_*+2\sqrt{N}\big)s\big\}
\\
&\le C(1+t)^C 
\exp\big\{N\eps_0 t -(\sqrt N+\gamma_*)a+\gamma_*\big(\gamma_*+2\sqrt{N}\big)s\big\}.
\enbal
\ee
(In the last line, we replaced $\shift(s)$ and $\shift(t)$ by their expression
\eqref{eq:mu}, and then used $s\le t$ for the logarithmic terms.)
Pick $c>0$, and then $\eps_0$ such that $N\eps_0\le c$. We see that
\eqref{may1968} holds if
\be\label{14jul1789}
ct-\gamma_* a+\gamma_*\big(\gamma_*+
2\sqrt{N}\big)s\le -ca-ct.
\ee
We now consider, till the end of this proof,
the case where \eqref{14jul1789} does not hold, i.e.\@ the case where 
\be\label{not14jul1789}
\gamma_*\big(\gamma_*+
2\sqrt{N}\big)s> (\gamma_*-c)a-2ct.
\ee
 In that
case, the bound $\psi(c)\le Ce^{\sqrt N x}$ that we used to derive
\eqref{20may1412x} is not good enough for $x>0$. We now show that if
$c>0$ is small enough, the
centering term in \eqref{20may1414x} satisfies
\be\label{mu2pos}
\shift(t)+2\gamma_* s+2\sqrt N t > ca.
\ee
Indeed, recall that $\shift(t)>-a$ and $\sqrt N>1$. If $t>a$, then
\eqref{mu2pos} is
obvious. If $t<a$, then \eqref{not14jul1789} and $t>s$ imply that
\begin{equation*}
\bal
\shift(t)+2\gamma_* s+2\sqrt N t 
&\ge -a + 2(\gamma_* +\sqrt N)s
\ge -a + \bigg[\frac{2(\gamma_* +\sqrt N) }{\gamma_*+2\sqrt
N}\bigg]\frac{\gamma_*-3c}{\gamma_*}a
\enbal
\end{equation*}
As the factor in square brackets is strictly larger than 1, it is possible
to choose $c>0$ small enough independent of $a$ such that \eqref{mu2pos}
holds.
Thus when $x<0$, \eqref{not14jul1789} implies
\begin{equation*}
\exp\Big[-\frac{(x - \shift(t) - 2\gamma_* s - 2\sqrt{N} t)^2}{4t}\Big]
\le \exp\Big[-\frac{(\shift(t) + 2\gamma_* s + 2\sqrt{N} t)^2}{4t}+\frac{c
a x}{2t}\Big].
\end{equation*}
Recalling \eqref{20may1430}, \eqref{20may1410x}, and \eqref{20may1414x},
we obtain
\begin{equation*}
e^{\sqrt N x} \rho^s(t,x) \le C \rho^s(t,0) \exp\Big(\frac{ca x}{2t}\Big) \quad \textrm{for } x < 0.
\end{equation*}
Therefore,
\be
\label{ontheleft}
\int_{-\infty}^0\psi(x) \rho^s(t,x) \d x \le C \frac t a \rho^s(t,0).
\ee

We now evaluate the integral over $x>0$, still assuming 
\eqref{not14jul1789}. In this region, we use the bound $\psi(x)\le C(1+x)$.
In the expression \eqref{20may1430} of $\rho^s(t,x)$, we
start by getting rid of the
logarithmic terms that appear in $\shift(t)$  and $\shift(s)$.
We employ \eqref{eq:gaussian mean bound} with some  $\eps$ to be chosen later.
Recalling that $N\eps_0\le c$, \eqref{20may1430} implies:
\be\label{tilderhos}
\rho^s(t,x)\le \tilde \rho^s(t,x)
\quad\text{with\quad} \tilde \rho^s(t,x)=
\frac C{\sqrt t} \exp\Big[2c
t-Nt -\gamma_*(a-\gamma_* s)-
(1-\eps)\frac{(x+a-2\gamma_* s)^2}{4t}\Big].
\ee
The term $2ct$ has two contributions.
On the one hand, $N\eps_0t \le ct$.
On the other, the logarithmic term in $\shift(s)$ is sublinear: $2\gamma_*/(3\lambda_*) \log(1+s) \le c t + C$.

Since $s\le \earlyt\le t_*=a/(2\sqrt N)$ and $\gamma_*\le\sqrt N$, we have
$a-2\gamma_*s\ge0$ and
\begin{equation*}
\tilde \rho^s(t,x)\le
\tilde\rho^s(t,0)\exp\Big[-(1-\eps)\frac{x^2}{4t}\Big]\qquad\text{for
$x>0$}.
\end{equation*}
Then
\be
\int_0^\infty \psi(x)\rho^s(t,x)\d x\le C\tilde\rho^s(t,0)\int_0^\infty(1+x)
\exp\Big[-(1-\eps)\frac{x^2}{4t}\Big]\d x
\le C(\sqrt t+t)\tilde \rho^s(t,0).
\label{ontheright}
\ee

Combining \eqref{ontheleft} and \eqref{ontheright}, we have shown that,
if \eqref{not14jul1789} holds, we have
\begin{equation*}
\int_\R\psi(x)\rho^s(t,x)\d x
\le
C(\sqrt t+t)\tilde\rho^s(t,0).
\end{equation*}

We now evaluate $\tilde\rho^s(t,0)$ and show that if $\epsilon$ is chosen
small enough, then there
exists $C>0$ and $c>0$ such that, for all $t>0$ and all $s\le t\wedge
\earlyt$,
\be\label{goalrhozero}
\tilde\rho^s(t,0)
     \le \frac C{\sqrt t}\exp\Big[   -(\sqrt N+c)a-ct\Big].
\ee
This will imply \eqref{may1968}, and then Lemma~\ref{lemmaaboutrho}
for the term with $\psi$.
We have from \eqref{tilderhos}
\begin{equation*}
\begin{aligned}
\tilde\rho^s(t,0)&= \frac C{\sqrt t} \exp\Big[2c
t-Nt -\gamma_*(a-\gamma_* s)-
(1-\eps)\frac{(a-2\gamma_* s)^2}{4t}\Big]
\\
&= \frac C{\sqrt t}   \exp\Big[2ct-Nt -\gamma_*(a-\gamma_* s)-
(1-\eps)\Big(\frac{a^2}{4t}-\frac{\gamma_*s}{t}(a-\gamma_*s)\Big)
\Big]
\\
&\le  \frac C{\sqrt t}   \exp\Big[2ct-(1-\eps)\Big(N\frac ta
+\frac{a}{4t}\Big)a -\gamma_*(a-\gamma_* s)\Big(1-(1-\eps)\frac st\Big) \Big]
\\&\le  \frac C{\sqrt t}   \exp\Big[2ct-(1-\eps)\theta\Big(\frac t a\Big)a
-\gamma_*(a-\gamma_* s)\Big(1-\frac st\Big) \Big]
\end{aligned}
\end{equation*}
with $\theta(\xi):=N\xi+1/(4\xi)$. Here we have used $a-\gamma_*s>0$.
Notice that $\theta(\xi)$ reaches its
minimum at  $\xi_*=t_*/a=1/(2\sqrt N)$ with
$\theta(\xi_*)=\sqrt N$. Let $t_0=(\earlyt+t_*)/2$. Notice also that $t_0/a$
does not depend on $a$ and that $t_0/a<\xi_*$. We choose $c$ small enough
that $\theta(t_0/a)>\sqrt N +5c$, and then $\eps$ small enough that
$(1-\eps)\theta(t_0/a) > \sqrt N+4c$. Then, for $t\le t_0$, we have
$\theta(t/a)\ge\theta(t_0/a)$ and
\begin{equation*}
\tilde\rho^s(t,0)\le  \frac C{\sqrt t}  \exp\Big[2ct-(1-\eps)\theta(t_0/a)a
\Big]\le \frac C{\sqrt t}\exp\Big[2ct-(\sqrt N+4c)a\Big],
\end{equation*}
where we used $s\le t$ and $s\le \earlyt\le t_*=a/(2\sqrt N)\le a/\gamma_*$.
Since $t\le t_0\le a$, this implies \eqref{goalrhozero}.

We now consider $t\in [t_0,a]$. Using 
$\theta(t/a)\ge\theta(\xi_*)=\sqrt N$ and  $s\le \earlyt=\earlyxi a$,
\begin{equation*}
\tilde\rho^s(t,0)\le  \frac C{\sqrt t}  \exp\Big[2ct-(1-\eps)\sqrt N a 
-a \gamma_*(1-\gamma_* \earlyxi)\Big(1-\frac {\earlyt}{t_0}\Big) \Big].
\end{equation*}
Notice that $\earlyt/t_0$ is independent of $a$ and strictly smaller than
one.
We make $c$ small enough that $\gamma_*(1-\gamma_*\earlyxi)(1-\earlyt/t_0)>5c$
and then $\eps$ small enough that $(1-\eps)\sqrt N\ge \sqrt N -c$.
Then, since $t\le a$, we obtain again the bound \eqref{goalrhozero}.

Finally, we consider $t\ge a$. As $\theta$ is convex, we have for $\xi\ge1$
\begin{equation*}
\theta(\xi)\ge\theta(1)+\theta'(1)(\xi-1)
=\frac12+\Big(N-\frac14\Big)\xi
\ge \sqrt N+\frac14+\Big(N-\sqrt N\Big)\xi
\end{equation*}
and
\begin{equation*}
\tilde\rho^s(t,0)\le \frac C{\sqrt t}\exp\Big[2ct-(1-\eps)\Big(\sqrt
N+\frac14\Big)a -(1-\eps)(N-\sqrt N) t \Big].
\end{equation*}
It is then clear that by taking $c$ and $\eps$ small enough, we have again
\eqref{goalrhozero}.

This concludes the proof that the first term in \eqref{eq:lemmaaboutrho}
is bounded as stated by Lemma~\ref{lemmaaboutrho}. In this proof, we only
used that $\psi(x)\le C \exp(\sqrt N x)$ for $x<0$ and $\psi(x)\le C(1+x)$
for $x\ge0$.
Since $\psi'(x)$ satisfies the same bound and $E$ is bounded,
the other terms in \eqref{eq:lemmaaboutrho} are bounded in the same way as
the first term and the
proof of Lemma~\ref{lemmaaboutrho} is complete.\hfill\qed

\subsection*{Proof of Lemma~\ref{lemmaaboutsigma}}

We now turn to the solution $\sigma$ of \eqref{eq:sigma-ev}:
\begin{equation*}
\pdr{\sigma}{t} = \pdrr{\sigma}{x} - \frac{3}{2\lambda_*(t+1)} \pdr{\sigma}{x}
-N(1 - \eps_0) \one_{(-\infty, -K)}(x) \sigma + N(1 - \eps_0) \one_{[-K, \infty)}(x) \rho, \quad \sigma(0, x) = 0.
\end{equation*}
Let us start by showing that 
\begin{equation*}
\int_\R \psi(x)\sigma(t,x) \d x    = O\left(e^{-(\sqrt{N} + c) a}\right).
\end{equation*}

We again represent $\sigma(t,x)$ via the Duhamel formula: 
\begin{equation}
\label{eq:sigma-Duhamel}
\sigma(t, x) = \int_0^t \sigma^s(t, x) \d s.
\end{equation}
Here, $\sigma^s(t,x)$, defined for $0<s<t$, is the solution to
\begin{equation}\label{eq:sigma-s}
\bal
&\pdr{\sigma^s}{t} = \pdrr{\sigma^s}{x} - \frac{3}{2\lambda_* (t + 1)}\pdr{\sigma^s}{x} - N (1 - \eps_0) \one_{(-\infty, -K)}(x) \sigma^s,~~~ t>s, \\
&\sigma^s(s, x) = N (1 - \eps_0) \one_{[-K, \infty)}(x) \rho(s, x).
\enbal
\end{equation}

Our first step will be to bound the initial condition in \eqref{eq:sigma-s} with the following Lemma.
\begin{lemma}
\label{lem:sigma-s-init}
There exist $C>0$, $c>0$, and $\kappa>0$ such that for all $s\ge 0$ and all $x\in \R$:
\be
\sigma^s(s,x)\le
C \exp\bigg[-(\sqrt{N}+c)a-cs-\kappa\min\left\{1,\frac{a}s\right\}
x_+ -\frac{x^2}{8s}\bigg]
\one_{[-K, \infty)}(x) .
\label{eq:sigma-s-init}
\ee
with $x_+=\max\{x,0\}$.
\end{lemma}
\begin{proof}
Recall from \eqref{eq:sigma-s} that $\sigma^s(s,x)=C\one_{[-K,
\infty)}(x)\rho(s,x)$ with $\rho(t,x)$
given by (\ref{20may1322}):
\begin{equation*}
\rho(t,x)=\int_0^{t\wedge \earlyt} \rho^{s}(t,x) \d s.
\end{equation*}
We use $\rho^{s}(t,x)\le \tilde\rho^s(t,x)$ with $\tilde\rho^s(t,x)$ given in
\eqref{tilderhos}.
It suffices to show that there exists $\kappa>0$ such that,
for all $s\le t\wedge\earlyt$,
\be\label{goalrhox}
\tilde\rho^s(t,x)\le \frac C{\sqrt t}\exp\bigg[-(\sqrt N+c)a-ct-\kappa
\min\left\{(1,\frac{a}t\right\}
x_+-\frac{x^2}{8t}\bigg]\quad\text{for $x>-K$}
\ee
Then $\rho(t,x)\le C\sqrt t \exp[\ldots]$, and the $\sqrt t$ can be absorbed
in the exponential by making $c$ smaller. We now show \eqref{goalrhox}.

We
first assume that $t\le\delta a$, for some small $\delta>0$ to be chosen.
For such small times, the Gaussian term in \eqref{tilderhos} controls
everything. We use the following result: for any $\eps\in[0,1/2)$, any
$\kappa\ge 0$, any $\lambda\ge 0$ and any $K\ge 0$, there exists $\delta'>0$ such
that, if $b$ is large enough,
\be\label{elemquad}
(1-\eps)\frac{(x+b)^2}{4t}\ge\frac{x^2}{8t}+\kappa x +\lambda b\qquad\text{for all
$t\le\delta' b$ and $x\ge-K$}.
\ee
Indeed, taking $\eps=0$ for simplicity, \eqref{elemquad} is equivalent to $x^2+2\beta x +\gamma\ge0$ with
$\beta=2b-4\kappa t$ and $\gamma=2b^2-8\lambda b t$. By making $\delta'$
small enough, we have $\beta\in[1.9b,2b]$ and $\gamma\in[1.9b^2,2b^2]$ for
$t\le \delta' b$.
Then, \eqref{elemquad} holds
for $x\ge-\beta+\sqrt{\beta^2-\gamma}$, where
$-\beta+\sqrt{\beta^2-\gamma}<(-1.9+\sqrt{4-1.9})b\approx-0.45 b$, which is
smaller than $-K$ for $b$ large enough.

Then, in \eqref{tilderhos}, we notice that $b:=a-2\gamma_*s\ge a-2\gamma_*
\earlyt=a(1-2\gamma_*\earlyxi)$, with $1-2\gamma_*\earlyxi>0$.
We use \eqref{elemquad} with this $b$ and $\lambda=(\sqrt N+c)/(1-2\gamma_*\earlyxi)$
to obtain \eqref{goalrhox} for $t\le\delta'b\le \delta a$ with
$\delta=\delta'(1-2\gamma_*\earlyxi)$.
(We wrote $2ct-Nt\le -ct$  in \eqref{tilderhos} after assuming
$c<N/3$, and we used $\kappa x\ge\kappa x_+-\kappa K$ for $x\ge-K$.)

We now turn to $t\ge\delta a$. From \eqref{tilderhos} and
\eqref{goalrhozero}, we have
\be\bal
\tilde\rho^s(t,x)&=\tilde\rho^s(t,0)\exp\bigg[-(1-\eps)\frac{x^2}{4t}-(1-\eps)\frac{a-2\gamma_*s}{2t}x\bigg]
\\&\le \frac C{\sqrt t}\exp\bigg[   -(\sqrt N+c)a-ct-\frac{x^2}{8t}-(1-\eps)\frac{a- 2\gamma_*s}{2t}x\bigg]
\enbal
\label{rhos_inter}
\ee
Pick $\kappa=(1-\eps)(1-2\gamma_*\earlyxi)/2$. Since $s\in[0,\earlyxi a]$ and
$t\ge\delta a$, we have
\begin{equation*}
  (1 - \eps) \frac{a - 2 \gamma_* s}{2 t} \in \left[\frac{\kappa a}{t},
\frac{1}{2 \delta}\right]. 
\end{equation*}
This
implies that
\begin{equation*}
(1-\eps)\frac{a-2\gamma_*s}{2t}x
\ge \frac{\kappa a}{t}x_+-\frac{K}{2\delta}
\ge \kappa \min\left\{1, \frac{a}{t}\right\} x_+-\frac{K}{2\delta}\text{\qquad for all }x\ge -K.
\end{equation*}
Using this bound in \eqref{rhos_inter} concludes the proof of \eqref{goalrhox} and of
Lemma~\ref{lem:sigma-s-init}.
\end{proof}

The solution to \eqref{eq:sigma-s} with the initial
condition~\eqref{eq:sigma-s-init} can be treated by Lemma~\ref{lem:super},
with $\kappa_-= 0$ and $\kappa_+=\kappa\min\{1,a/s\}.$
For instance,
\begin{equation*}
  \int_\R\psi(x)\sigma^s(t,x)\,\d x\le Ce^{-(\sqrt N+c)a-cs} \max\Big\{\Big(\frac sa\Big)^2,1\Big\}.
\end{equation*}
Since $a \geq 1$, we have $\max\big\{(s/a)^2,1\big\} \leq C(\eps) e^{\eps s}$ for any $\eps > 0$.
Thus, we can absorb $\max\big\{(s/a)^2,1\big\}$ into the exponential factor $e^{-cs}$ by reducing $c$.
After this operation, Lemma~\ref{lem:super} yields:
\be
\label{eq:sigma-s-integrals}
\begin{aligned}
&\int_\R\psi(x)\sigma^s(t,x)\,\d x\le Ce^{-(\sqrt N+c)a-cs},
\\
&\int_\R\psi'(x)\sigma^s(t,x)\,\d x\le C e^{-(\sqrt N+c)a-cs} (t-s+1)^{-1/2},
\\
&\int_\R|E(t,x)|
\psi(x)\sigma^s(t,x)\,\d x\le Ce^{-(\sqrt N+c)a-cs} (t+1)^{-1/2}(t-s+1)^{-3/2}.
\end{aligned}
\ee
In light of the Duhamel formula~(\ref{eq:sigma-Duhamel}), we must integrate \eqref{eq:sigma-s-integrals} over $s \in [0, t]$.
Taking $\al \in \{1/2, 3/2\}$ in \eqref{eq:poly-exp-integral}, \eqref{eq:sigma-Duhamel} and \eqref{eq:sigma-s-integrals} imply Lemma~\ref{lemmaaboutsigma}.\hfill\qed

\appendix
\section{The proof of Lemma~\ref{lem:super}}\label{sec:appendix}

To prove Lemma~\ref{lem:super}, we argue that the absorption on
$(-\infty,-K)$ in \eqref{eq:super} acts similarly to the Dirichlet 
boundary condition.

We begin by constructing a family of supersolutions based on the Dirichlet problem.
Recall that we defined
\begin{equation*}
  \varphi(\lambda, x) = \frac{x}{\lambda^{3 /2}} \exp\left(-\frac{x^2}{4\lambda}\right).
\end{equation*}
Now let $v(t,x;s)$, with $s\ge 0$,  be the solution to
\begin{equation}\label{eq:Dirichlet}
\bal
&    \pdr{v}{t} = \pdrr{v}{x} - \frac{3}{2(t + s + 1)} \pdr{v}{x},~~~~ \textrm{for $t>0$ and $x>0$, } \\
&v(t, 0) = 0,~~~~~~~~~~~~~~~~~~~~~~~~~~~~\textrm{for $t>0$,}  \\
 &v(0,x) = \varphi(\lambda, x) ~~~~~~~~~~~~~~~~~~~~~ \textrm{for $x>0$. }
\enbal
\end{equation}

We define the ``$\lambda$-adapted'' self-similar variables
\begin{equation*}
  \tau = \log \left(\frac{t + \lambda}{t}\right) \And \eta = \frac{x}{\sqrt{t + \lambda}}.
\end{equation*}
In these coordinates, $v$ satisfies
\begin{equation*}
  \pdr{v}{\tau} = \pdrr{v}{\eta} + \frac{\eta}{2} \pdr{v}{\eta} - m(\tau) \pdr{v}{\eta}, \quad v(0, \eta) = \varphi(1, \eta),
\end{equation*}
with a drift  
\begin{equation*}
 m(\tau) = \frac{3}{2} \frac{\sqrt{\lambda}e^{\tau/2}}{\lambda(e^\tau - 1) + s + 1}.
\end{equation*}
We wish to argue that the drift term $m \partial_\eta v$ is negligible, so that $\varphi(\lambda + t, x)$ is an approximate solution to~\eqref{eq:Dirichlet}.
To this end, note that if $s\ge c_0\lambda$ with $c_0>0$ fixed, we have
\begin{equation*}
  m(\tau) \leq Cs^{-1/2} \ll 1.
\end{equation*}
Combining Lemma~2.2 in \cite{HNRR} with methods from \cite{NRR1}, we have the following analogue of Lemma~\ref{lem:Dirichlet-lower}. 
\begin{lemma}\label{lem:Dirichlet}
Fix $c_0 > 0$ and suppose that $s \geq c_0 \lambda$.
There exists $C(c_0) > 0$ such that if $s \geq C$, then the solution $v(t,x;s)$ to (\ref{eq:Dirichlet}) has the form
\begin{equation*}
  v(t, x) = \varphi(t + \lambda, x) + (t + \lambda)^{-3/ 2} v_0(t,x;s),
\end{equation*}
with 
\begin{equation*}
\|\ubar v_0(s)\|_{\m C_w^1}+\|\ubar v_0(s)\|_{L_w^2}\le Cs^{-1/2}
\end{equation*}
for the weight $w=\exp\left[ {x^2}/{(6(t - s + 1))}\right]$.
\end{lemma}

This lemma also permits the construction of supersolutions for more complicated equations.
If we have an absorbing potential on the left rather than a Dirichlet condition, we can join $v$ to a decaying exponential.
We illustrate the construction with a simple example.
Consider the equation
\begin{equation}
  \label{eq:potential-example}
  \pdr{\nu}{t} = \pdrr{\nu}{x} - \frac{3}{2(t + 1 + s)} \pdr{\nu}{x} - \one(x<2) \nu.
\end{equation}
It is easy to check that 
\begin{equation*}
  A_\lambda(t,x)= \frac 1 2 (t+\lambda)^{-3/ 2}\exp\left(\frac{x}{2}\right)
\end{equation*}
is a supersolution to (\ref{eq:potential-example}) on $(-\infty, 2]$, provided $\lambda \geq 2$.
We can glue $A_\lambda(t,x)$ on the left to the solution to the Dirichlet problem (\ref{eq:Dirichlet}) on the right, with no
potential, to form a  global super-solution to~(\ref{eq:potential-example}).
For the hybrid to be a supersolution itself, we need its slope to decrease at the joint, which
ensures that in a neighborhood of the joint, it is the minimum of two supersolutions, and thus a supersolution.
We must thus take some care to achieve this.

To emphasize the dependence on $\lambda$, let $v_\lambda$ denote the solution to \eqref{eq:Dirichlet}.
If $\lambda \geq 2$, the graphs of~$A_\lambda(0,x)$ and~$\varphi_\lambda(x)$
intersect twice in $[0, 2]$.
By Lemma~\ref{lem:Dirichlet}, the same is true of $A_\lambda(t,x)$ and~$v_\lambda(t,x;s)$ for $s$ sufficiently large.
Therefore, 
the hybrid function
\begin{equation*}
  \Psi(t, x) =
  \begin{cases} A_\lambda(t,x) 
    & \textrm{if } x \leq j(t)\\
    v_\lambda(t, x; s) & \textrm{if } x > j(t),
  \end{cases}
\end{equation*}
where $j(t)$ is the rightmost point of intersection, is a supersolution to \eqref{eq:potential-example}.

Now, consider $w$ as in Lemma~\ref{lem:super} with $s \geq C$ fixed.
Let $\bar{w}$ denote the solution to \eqref{eq:super} and \eqref{eq:super-initial} with inequalities replaced by equalities.
Then, by the comparison principle, we have $w \leq \bar{w}$, and it suffices to control $\bar{w}$.
Next shift $\bar{w}$ so that it starts at time $0.$
To bound $\bar{w}$ from above, we will piece together many supersolutions of a hybrid form.

In the following, we assume that~$\lambda$ and $s$ satisfy the hypotheses of Lemma~\ref{lem:Dirichlet} and $s \geq C' \geq C$ for some $C'$ to be determined.
We say a function is hybrid if it is a spatial shift and constant multiple of the following form:
\begin{equation*}
\Phi(t, x) =
  \begin{cases}
    B (t + \lambda)^{-3 /2} e^{-b x} & \textrm{for } x \leq j(t),\\
    v_\lambda(t, x; s) & \textrm{for } x > j(t).
  \end{cases}
\end{equation*}
Here $B, C' > 0$, $b < \al$, and $j \colon [0, \infty) \to \R_+$ (uniformly bounded) are chosen so that $\Phi$ is continuous and $\partial_x\Phi$ decreases at the joint $x = j(t)$.
If $\lambda$ is sufficiently large (depending on $b$), $\Phi$ is a supersolution to \eqref{eq:super}.
Let $\m H$ denote the set of hybrids.

Notice that
\begin{equation*}
  \int_{0}^\infty \varphi(\lambda, x) \d x \sim \lambda^{-\frac 1 2} \And \int_{\R_+} x \varphi(\lambda, x) \d x \sim 1.
\end{equation*}
If $\Phi$ is a hybrid with parameter $\lambda$, Lemma~\ref{lem:Dirichlet} ensures that it resembles $\varphi(t + \lambda, x)$ on $\R_+$,
and it is straightforward to check that
\begin{equation}
  \label{eq:hybrid-ints}
  \begin{aligned}
    &  \int_\R \psi(x) \Phi (t,x) \d x\leq C,\\
    &  \int_\R \psi' (x)\Phi (t,x) \d x\leq C \lambda^{-1 /2},\\
    & \int_\R E(t,x) \psi(x) \Phi (t,x) \d x \leq C (t+1)^{-1/ 2} (t + \lambda)^{-3/ 2}.
  \end{aligned}
\end{equation}

To control the evolution of $\bar{w}$, we use a sum of hybrid supersolutions to cover different spatial regions of the initial condition $\bar{w}(0, x)$.
On $\R_-$, we choose $\Phi_- \in \m H$ of the form 
\begin{equation*}
  \Phi_-(t,x)= B (t + \lambda)^{-3/ 2} e^{-b x}
\end{equation*}
on the left, with some $b \in (\kappa_-, \al)$ and $\lambda \geq 1$.
On the right, $\Phi_-$ has the form $D v_{\lambda}$ for some $D > 0$.
Taking $B, D \gg 1$, we can ensure that $\bar{w}(0, x) \leq H_-(0,x)$ for $x\le 1$. 

We are thus left with the decaying tail of $\bar{w}(0,x)$ for $x\ge 1$. 
We cover this tail by a sequence of ever-wider Gaussians.
Noting that $\|\varphi(\lambda, \anon)\|_\infty \sim \lambda^{-1}$, we use the normalized functions $\lambda \varphi(\lambda, x)$ as building blocks.
In particular, there exists universal $D_+ > 0$ such that
\begin{equation}
  \label{eq:sum-init}
  \bar{w}(0, x) \leq \sum_{k = 0}^{\log_2 s + 1} D_+ \exp\left(-\kappa_+ 2^k \right) 2^k \varphi(2^k, x) \ForAll x \geq 1.
\end{equation}
Note that we are able to truncate the sum at $k = \log_2 s + 1$ because $\bar{w}(0, x)$ is bounded by the Gaussian~$\exp\{-{x^2}/{(8s)}\}$.
Using \eqref{eq:sum-init}, we construct hybrid supersolutions $\Phi_k \in \m H$ equal to a shifted multiple of $v_{2^k}(t, x; s)$ on the right.
Note that the condition $k \leq \log_2 s + 1$ implies $\lambda = 2^k \leq 2s$.
We thus satisfy the hypotheses of Lemma~\ref{lem:Dirichlet}.

By the comparison principle, we have
\begin{equation*}
  \bar{w} \leq \Phi_- + \sum_{k=0}^{\log_2 s + 1} \Phi_k.
\end{equation*}
The sum of the coefficients in \eqref{eq:sum-init} is bounded by $C \kappa_+^{-2}$.
Therefore, \eqref{eq:hybrid-ints} implies
\begin{align*}
&  \int_\R \psi (x)\bar{w}(t,x) \d x \leq C \max\{\kappa_+^2, 1\},\\
 & \int_\R \psi' (x)\bar{w}(t,x) \d x \leq C \max\{\kappa_+^2, 1\} (t  + 1)^{-1/ 2},\\
 & \int_\R e^{-\delta \abs{x}} \psi(x) \bar{w}(t,x) \d x \leq C \max\{\kappa_+^2, 1\} (t + 1)^{-3/ 2}.
\end{align*}
Now, \eqref{eq:psi-super}--\eqref{eq:wave-super} follow, after recalling  the time change $w(t, x) \leq \bar{w}(t - s, x)$ due to our temporal shift of $\bar{w}$.

We have yet to handle $s \leq C'$.
We first note that $\Phi_k$ only depends on $\kappa_+$ through a scalar multiple.
Hence $C'$ is independent of $\kappa_+$.
When $s \leq C'$, we can simply forget about the absorbing potential in \eqref{eq:super}, and allow our solution to evolve under a (drifting) heat equation until time $C'$.
This will certainly be a supersolution for \eqref{eq:super}.
Since the solution at time $C'$ will still be bounded by a Gaussian, we can run our argument above once we reach time $C'$.
This completes the proof of Lemma~\ref{lem:super}.\hfill\qed

\section{The differentiability of the Bramson shift}
\label{sec:s-differentiable}

In this appendix, we justify \eqref{eq:z-limit-formal}.
\begin{lemma}
  \label{lem:s-differentiable}
  The shift $s(y, a)$ is differentiable with respect to $y$.
  Moreover, for all $a > 0$,
  \begin{equation}
    \label{eq:commute-lim-deriv}
    \lim_{t \to \infty} z\big(t, x + m(t); a\big) = -\partial_y s(0, a) U'(x - \bar{x}_0)
  \end{equation}
  uniformly in $x \in \R$.
\end{lemma}

\begin{proof}
  We fix $a > 0$ throughout the proof, and allow constants to implicitly depend on $a$.
  To establish differentiability, we prefer to work with a family of initial data defined on an open interval around $y = 0$.
  We therefore define
  \begin{equation}
    \label{eq:initial-y-full}
    u(0, x; y, a) = \one_{(-\infty, 0]}(x) + \one_{(-\infty, y-a]}(x) - \one_{(-\infty, -a]}(x) \quad \trm{for } y \in \R.
  \end{equation}
  This coincides with the definition in \eqref{eq:initial-y} when $y < 0$, and satisfies
  \begin{equation}
    \label{eq:y-deriv}
    \partial_y u(0, x; y, a) = \delta(x - y + a)
  \end{equation}
  in the distributional sense for all $y \in \R$.
  We prove Lemma~\ref{lem:s-differentiable} at $y = 0$; the general case is similar.
  We are thus free to restrict our attention to $y \in (-1, 1)$.

  Again, we allow \eqref{eq:initial-y-full} to evolve under the Fisher--KPP equation \eqref{20mar1914} and obtain a solution $u(t, x; y, a)$.
  We note that $u(0, x; y, a)$ exceeds $1$ on a compact interval when $y > 0$, so we must extend the reaction $f$ beyond the interval $[0, 1]$ to define the evolution.
  As mentioned earlier, standard parabolic estimates show that $z(t, x; y, a) = \partial_y u(t, x; y, a)$ exists and solves
  \begin{equation}
    \label{eq:z-full}
    \partial_t z(t, x; y, a) = \partial_x^2 z(t, x; y, a) + f'\big(u(t, x; y, a)\big) z(t, x; y, a), \quad z(0, x; y, a) = \delta(x - y + a).
  \end{equation}
  Also, there exists $s(y, a) \in \R$ such that
  \begin{equation}
    \label{eq:u-y-limit}
    u\big(t, x + m(t); y, a\big) \to U\big(x - s(y, a)\big) \quad \trm{as } t \to \infty
  \end{equation}
  uniformly in $x \in \R$.
  When $y=0$, Proposition~\ref{prop:z-limit} shows that $z$ converges to a multiple of $U_0'$ in the $m$-moving frame as $t \to \infty$.
  The argument doesn't rely on the particular value of $y$, so in fact there exists $M(y, a) > 0$ such that
  \begin{equation}
    \label{eq:z-y-limit}
    z\big(t, x + m(t); y, a\big) = \partial_yu\big(t,x+m(t);y,a\big) \to -M(y, a) U'\big(x - s(y, a)\big) \quad \trm{as } t \to \infty
  \end{equation}
  uniformly in $x \in \R$.

  For the moment, let us fix $x$ and view $u\big(t, x + m(t); y, a\big)$ as a sequence of functions of $y$ indexed by time $t$.
  By \eqref{eq:u-y-limit}, the sequence $u$ converges pointwise in $y$ as $t \to \infty$ to the limit $U\big(x - s(y, a)\big)$.
  By \eqref{eq:z-y-limit}, the derivative $\partial_y u$ converges pointwise in $y$ as $t \to \infty$ to the limit $-M(y, a) U'\big(x - s(y, a)\big)$.
  We recall a standard result from real analysis: if the convergence \eqref{eq:z-y-limit} is in fact \emph{uniform} in $y \in (-1, 1)$, then the limit of $u$ is differentiable in $y$ and its derivative is the limit of $\partial_y u$.
  Assuming this holds, we obtain
  \begin{equation}
    \label{eq:commute-lim-deriv-full}
    \begin{aligned}
      \lim_{t \to \infty} z\big(t, x + m(t); y, a\big) &= -M(y, a) U'(x - s(y, a))\\
      &= \partial_y\left[U\big(x - s(y, a)\big)\right] = -\partial_y s(y, a) U\big(x - s(y, a)\big).
    \end{aligned}
  \end{equation}
  Here, we have used the fact that $U\big(x - s(y, a)\big)$ is differentiable in $y$ if and only if $s$ is.
  Taking $y = 0$ in \eqref{eq:commute-lim-deriv-full} and using the $x$-uniformity in \eqref{eq:z-y-limit}, we obtain \eqref{eq:commute-lim-deriv}.
  Thus, it suffices to show that \eqref{eq:z-y-limit} holds uniformly in $y$ for each $x \in \R$ and $a > 0$.

  To establish uniformity in $y$, we rely on a quantitative form of \eqref{eq:z-y-limit}.
  The main result of \cite{NRR2} shows that the error in \eqref{20mar1918} is of order $t^{-1/2}$; \emph{Cf.} \eqref{20may1932}.
  As mentioned earlier, \eqref{eq:z-full} is structurally similar to the Fisher--KPP equation \eqref{20mar1914}: after an appropriate shift and exponential tilt, both resemble the Dirichlet heat equation on $\R_+$.
  The argument of \cite{NRR2} only requires this general structure, so a straightforward adaptation of Theorem~1.3 in \cite{NRR2} implies
  \begin{equation*}
    z\big(t, x + m(t); y, a\big) = -M(y, a) U'\big(x - s(y, a)\big) + R(t, x; y, a)
  \end{equation*}
  for a remainder $R$ satisfying
  \begin{equation}
    \label{eq:z-y-error}
    \abs{R(t, x; y, a)} \leq C(a) e^{-c\abs{x}} (t + 1)^{-1/2}.
  \end{equation}
  
  Crucially, this bound is independent of $y$.
  This is because $z$ only depends on $y$ through its initial data and the potential $f'\big(u(t, x; y, a)\big)$ in \eqref{eq:z-full}.
  Now, the initial conditions $z(0, x; y, a) = \delta(x - y + a)$ are supported in a common compact interval $[-a-1, -a + 1]$ and have the same mass.
  Likewise, the potentials $f'\big(u(t, x; y, a)\big)$ are monotone in $y$, and are thus controlled by the endpoints $y = \pm 1$.
  In particular, they are uniformly positive on the left and decay uniformly exponentially on the right.
  In this sense, they approximate the Dirichlet heat equation on $\R_+$ uniformly-well in $y.$
  
  The estimates in \cite{NRR2} are easily seen to be uniform in initial data and potentials of this form, so the error estimate \eqref{eq:z-y-error} is independent of $y$.
  It follows that \eqref{eq:z-y-limit} holds uniformly in $y$ for fixed $x$.
  As noted above, the lemma follows.  
\end{proof}


\begin{thebibliography}{99}
\bibliographystyle{plain}

\bibitem{ABBS} E. A\"{i}d\'ekon, J. Berestycki, \'E. Brunet, Z. Shi, {Branching Brownian motion seen from its tip}, {\it Probab. Theory Relat. Fields} {\bf 157}, 2013, 405--451.

\bibitem{ABK1} L.-P. Arguin, A. Bovier, and N. Kistler, {Poissonian statistics in the extremal process of branching Brownian motion}, {\it Ann. Appl. Probab.} {\bf 22}, 2012,  
1693--1711.

\bibitem{ABK2} L.-P. Arguin, A. Bovier, and N. Kistler, {The extremal process of branching Brownian motion}, {\it Probab. Theory Relat. Fields} {\bf 157}, 2013, 535--574.

\bibitem{Bramson1} 
M. D. Bramson, {Maximal displacement of branching Brownian motion}, 
{\it Comm. Pure Appl. Math.} {\bf 31}, 1978, 531--581.

\bibitem{Bramson2} 
M. D. Bramson, {Convergence of solutions of the Kolmogorov equation
  to travelling waves}, {\it Mem. Amer. Math. Soc.} {\bf 44}, 1983. 

\bibitem{BD1} \'E. Brunet and B. Derrida. {Statistics at the tip of a branching random walk and the delay of traveling waves}, {\it Eur. Phys. Lett.} {\bf 87}, 2009, 60010.

\bibitem{BD2} \'E. Brunet and B. Derrida.  {A branching random walk seen from the tip}, {\it Jour. Stat. Phys.} {\bf 143}, 2011, 420--446.

\bibitem{CHL} A. Cortines, L. Hartung, and O. Louidor,
{The structure of extreme level sets in branching Brownian motion},
{\it Ann. Probab.}
	{\bf 47}(4), 2019, 2257--2302.
 
\bibitem{Fisher}
R. A. Fisher, The wave of advance of advantageous genes, {\it Ann. Eugen.} {\bf 7}, 1937, 355--369. 

\bibitem{Graham}
C. Graham, Precise asymptotics for Fisher--KPP fronts, {\it Nonlinearity }
{\bf 32}, 2019, 1967--1998. 

\bibitem{HNRR}
F. Hamel, J.  Nolen, J.-M. Roquejoffre and L. Ryzhik, A short proof of the logarithmic Bramson correction in Fisher-{KPP} equations,
{\it Netw. Heterog. Media} {\bf 8}, 2013, 275--289.

\bibitem{KPP} A. N. Kolmogorov, I. G. Petrovskii and N. S. Piskunov, {\'Etude de 
l'\'equation de la diffusion avec croissance de la quantit\'e de mati\`ere et son 
application \`a un probl\`eme biologique}, {\it Bull. Univ. \'Etat Moscou, S\'er. 
Inter.~A} {\bf 1}, 1937, 1--26.

\bibitem{LS} S. P. Lalley and T. Sellke, { A conditional limit theorem for the frontier of a branching Brownian motion}. {\it Annals of Probability} {\bf 15},
1987, 1052--1061.

\bibitem{Lau} K.-S. Lau, {On the nonlinear diffusion equation of Kolmogorov, Petrovskii and Piskunov}, {\it J.~Diff. Eqs.} {\bf 59}, 1985, 44--70.

\bibitem{McK}   H. P. McKean, {Application of Brownian motion 
to the equation of Kolmogorov-Petrovskii-Piskunov}, {\it Comm. Pure Appl. Math.} {\bf 28} 
1975, 323--331.

\bibitem{MRR} L. Mytnik, J.-M. Roquejoffre and L. Ryzhik, 
Fisher--KPP equation with small data and the extremal process of branching Brownian motion,
\emph{arXiv e-prints}, 2020, \href{https://arxiv.org/abs/2009.02042}{2009.02042}.

\bibitem{NRR1}
J.  Nolen, J.-M. Roquejoffre and L. Ryzhik,
Convergence to a single wave in the Fisher--KPP equation, {\it Chin. Ann. Math. Ser. B} {\bf 38}, 2017, 629--646.
 
\bibitem{NRR2}
J.  Nolen, J.-M. Roquejoffre and L. Ryzhik,
Refined long-time asymptotics for Fisher--KPP fronts, {\it Comm. Contemp. Math.}, 2018, 1850072. 
 
\bibitem{Roberts} M. Roberts, 
{A simple path to asymptotics for
the frontier of a branching Brownian motion}, {\it Ann. Prob.} {\bf 41}, 2013, 
3518--3541.

\bibitem{Uchiyama} K. Uchiyama, The behavior of solutions of some nonlinear 
diffusion equations for large time, {\it J.~Math. Kyoto Univ.} {\bf 18}, 1978, 453--508.

\end{thebibliography}
\end{document}